\documentclass[11pt,dvipsnames]{amsart}
 
\usepackage{amsfonts,amssymb,lscape,mathrsfs,graphicx}
\usepackage{amsthm,bbm}
\input xy
\xyoption{all}
\usepackage{mathtools}
\usepackage{tikz}
\usepackage{color,xcolor}
\usepackage[margin=1.15in]{geometry}
\usepackage[hypertexnames=false,backref=page,pdftex,
 	pdfpagemode=UseNone,
 	breaklinks=true,
 	extension=pdf,
 	colorlinks=true,
 	linkcolor=blue,
 	citecolor=red,
 	urlcolor=blue,
 ]{hyperref}

\usetikzlibrary{arrows,matrix}

\newtheorem{theorem}{Theorem}
\numberwithin{theorem}{section}

\theoremstyle{definition}
\newtheorem{proposition}[theorem]{Proposition}
\newtheorem{lemma}[theorem]{Lemma}
\newtheorem{setup}[theorem]{Setup}
\newtheorem{corollary}[theorem]{Corollary}

\newtheorem{definition}[theorem]{Definition}
\newtheorem{remark}[theorem]{Remark}
\newtheorem{example}[theorem]{Example}
\newtheoremstyle{citing}
  {}
  {}
  {\itshape}
  {}
  {\bfseries}
  {\textbf{.}}
  {.5em}
  {\thmnote{#3}}

{\theoremstyle{citing}
\newtheorem*{custom}{}}

\newcommand{\vth}{\vartheta}
\newcommand{\vol}{{\mathrm{vol}}}
\newcommand{\RR}{\mathbb{R}}

\newcommand{\sB}{{\mathcal B}}

\newcommand{\sF}{{\mathcal F}}

\newcommand{\sK}{{\mathcal K}}

\newcommand{\sP}{{\mathcal P}}

\newcommand{\scrE}{{\mathscr E}}

\newcommand{\C}{{\mathbb C}}

\newcommand{\E}{{\mathbb E}}

\newcommand{\N}{{\mathbb N}}
\renewcommand{\P}{{\mathbb P}}

\newcommand{\R}{{\mathbb R}}

\newcommand{\Z}{{\mathbb Z}}

\newcommand{\abs}[1]{{\left|#1\right|}}

\newcommand{\eps}{\varepsilon}

\newcommand{\id}{{\rm id}}

\newcommand{\into}{{\, \hookrightarrow\,}}
\newcommand{\isom}{{\ \cong\ }}

\newcommand{\ohne}{{\ \setminus \ }}

\newcommand{\pt}{{\operatorname{pt}}}

\newcommand{\supp}{\operatorname{supp}}

\renewcommand{\to}[1][]{\xrightarrow{\ #1\ }}

\newcommand{\veps}{\varepsilon}
\newcommand{\vphi}{\varphi}

\newcommand{\wh}[1]{{\widehat{#1}}}

\newcommand{\klmn}{{k,m,n}}
\newcommand{\BC}{{\sB}} 
\newcommand{\BCwas}{{\hat\BC_p}} 
\newcommand{\BCbot}{{\hat\BC_{\infty}}} 
\newcommand{\lone}{{\ell_1}}

\renewcommand{\S}{\mathbb{S}}
\renewcommand{\k}{\mathbf{k}}
\newcommand{\norm}[1]{\left\lVert#1\right\rVert}
\newcommand{\PH}{{\mathrm{PH}}}
\newcommand{\one}[1]{{\mathbbm{1}_{#1}}}

\newcommand{\Rplus}{{\R_{\geq 0}}}

\begin{document}
\title{Geometric and Probabilistic Limit Theorems in Topological Data Analysis}
\author{Sara~Kali\v{s}nik, Christian Lehn, and Vlada Limic}
\begin{abstract}
We develop a general framework for the probabilistic analysis of random finite point clouds in the context of 
 topological data analysis. 
We extend the notion of a barcode 
of a finite point cloud to compact metric spaces. Such a barcode lives in the completion 
of the space of barcodes with respect to the bottleneck distance, which is quite natural from an analytic point 
of view. As an application we prove that the barcodes of i.i.d. random variables 
sampled from a compact metric space converge to the barcode of the support of their distribution when the number 
of points goes to infinity. We also examine more quantitative convergence questions for uniform sampling from 
compact manifolds, including expectations of transforms of barcode valued random variables in Banach spaces. 
We believe that the methods developed here will serve as useful tools in studying more sophisticated questions 
in topological data analysis and related fields. 
\end{abstract}


%
%
%

\maketitle
\let\languagename\relax
\tableofcontents

\section{Introduction}\label{section intro}
Topological Data Analysis (TDA) is a fast growing field whose aim is to provide a set of new topological and geometric tools for analyzing data. One of the most widely used tools is persistent homology. The ideas behind persistent homology can be traced back to the works of Patrizio Frosini~\cite{Frosini} on size functions, and of Vanessa Robins~\cite{Robins} on using experimental data to infer the topology of attractors in dynamical systems, though the method only gained prominence with the pioneering works of Edelsbrunner, Letscher and Zomorodian~\cite{elz-tps-02} and Carlsson and Zomorodian~\cite{ZC}. 
Persistent homology has been used to address problems in fields ranging from sensor networks~\cite{VinEvader, adams}, medicine~\cite{Ferri, ADCOCK201436}, neuroscience~\cite{Chung2009, Curto2013, Giusti_Pastalkova_Curto_Itskov_2015}, as well as imaging analysis~\cite{Klein}.

The input of persistent homology is usually a point cloud, i.e. a finite metric space. 
Since finitely many points do not carry any nontrivial topological information, the idea is to consider the homology of
 thickenings of this point cloud in order to deduce information about the data or the distribution it is sampled from.
The output is a barcode, i.e. a multiset of intervals,
where each interval (``bar'') represents a topological feature present at parameter values specified by the interval.
This space of barcodes $\BC$ comes equipped with natural metrics, for example the Wasserstein and the Bottleneck distance.

The present paper grew out of an attempt to understand how some of the fundamental aspects of persistent homology \emph{and probability theory} could interact in order to
 allow for further statistical applications.
Here and in the rest of the introduction we will present some of our key results. 

Firstly, we wish to extend the notion of a barcode from finite sets to compact sets. 
This is done in 

\begin{custom}[Proposition \ref{proposition barcodes for compact metric spaces}] Let $k \in \N_0$ be a nonnegative integer, and let $M$ be a metric space.
 Then for every compact set $K\subset M$ there is a barcode $\beta_k(K) \in \BCbot$ such that $K \mapsto \beta_k(K)$ is a $1$-Lipschitz map from the space of
 compact subsets of $M$, equipped with the Hausdorff distance, to the completion $\BCbot$ of the barcode space $\BC$, with respect to the bottleneck distance. 
\end{custom}

This result can also be obtained from  the main theorem of  \cite{CSEH07}. It was later explicitly stated and proved in~\cite[Proposition 5.1]{Chazal2014} and relied on a measure theoretic approach to persistent homology introduced in~\cite{chazal:hal-01330678}. For completeness, we include a simple, conceptually clear, and self-contained proof. See Remark \ref{remark barcodes for totally bounded spaces} for an extension to totally bounded spaces.

Since we use a limiting procedure to define $\beta_k$ on compact subsets of $M$, the barcode $\beta_k(K)$ has to live in the completion $\BCbot$, which is a natural space 
for doing analysis with barcodes.

\enlargethispage{\baselineskip}

Suppose now that  the point cloud is obtained by sampling independent and identically distributed (i.i.d.)  points from an unknown distribution with compact support $C$.
The following question seems very natural, and it is somewhat surprising that it has not yet been answered:

\begin{center}
	\emph{What happens to the barcode as we sample more and more such points?}
	\end{center} 
	
\noindent	
In Section \ref{section approximation two}, we provide the following intuitive answer. 
\begin{custom}[Theorem \ref{theorem almost sure limit of number of points}] 
Let $M$ be a metric space, $X_1,X_2, \ldots$ be i.i.d.~$M$-valued random variables, and $k\in \N_0$. 
Define $P_n=\{X_1,X_2, \ldots,X_n\}$.
If the distribution of the $X_i$ has support equal to a compact subset $C \subset M$, then
\[
\beta_k(C) = \lim_{n\to \infty} \beta_k(P_n) \ \ \textrm{ almost surely.}
\]
\end{custom}

In the stochastic setting we also address questions about the mean and deviation from (or concentration about) the mean. 
For this discussion we consider random variables taking values in some Banach space. 
Starting from a barcode valued random variable $\beta$ (e.g.~$\beta=\beta_k(P_n)$ as above), one can obtain a Banach space 
valued random variable, as in~\cite{algfn, Bub15, polonik, Kal18, Reininghaus_2015_CVPR} and many others. In this paper we use the functions proposed by~\cite{Kal18} as a primary example, though the results are stated in full generality for Lipschitz continuous maps
\[
 T: \BCbot \to V
\]
from the completed space of barcodes to some  space $V$. Mapping to a Banach space is necessary to be able to talk about stochastic quantities such as expectation. But even more importantly, in order to use the information contained in the barcode using machine learning algorithms, one needs to produce a vector valued output. As mentioned above, there is a plethora of methods to produce such an output and our setup covers all of them, the only hypothesis being that the map $T$ is Lipschitz continuous. Note that it has been shown in \cite{CarriereB19} that one cannot embed the space of barcodes with finitely many functions which is why we stress the (infinite dimensional) Banach setup.

The study of probabilistic properties of $T\circ \beta$ naturally leads to the law of large numbers and a central limit theorem, as Bubenik first observed in~\cite{Bub15}. 
They can be formulated as follows.
\begin{custom}[Theorems \ref{theorem lln for barcodes} and \ref{theorem clt for barcodes}]---

\begin{itemize}
 \item Let $T: \BCbot \to V$ be a continuous map from the (bottleneck completed) space of barcodes to a separable Banach space $V$ and let $\{X_i\}_{i\in \N}$ be an
 i.i.d.~sequence of $\BCbot$-valued random variables such that $\E[\norm{T(X_1)}]< \infty$. Then the sequence $(S_n)_n$ of empirical means
$$
S_n := \frac{T(X_1) + \ldots + T(X_n)}{n}
$$
converges almost surely to $\E[T(X_1)]$. 
\item Suppose that in addition $\E[{T(X_1)}]=0$ and $\E[\norm{T(X_1)}^2]< \infty$, and let $S_n$ be as above. 
If $V$ is of type $2$, then $(\sqrt{n} S_n)_n$ converges weakly to a Gaussian random variable with the covariance structure of $T(X_1)$.
\end{itemize}
\end{custom}

The reader may be concerned about the vacuousness of the just stated result due to its rather abstract setting.
We respond by addressing the important situation of barcodes of compact metric spaces (in particular including that of point clouds sampled from a distribution with compact support).

\begin{custom}[Theorem \ref{theorem hypothesis of lln and clt for compact}]
Let $M$ be a metric space, $K(M)$ the complete metric space of all compact subsets of $M$ with the Hausdorff distance, and $X$ a random variable taking
 values in $K \in K(M)$.
Consider the $k$-th barcode map $\beta_k:K\left(M\right)\to \BCbot$ and let $T:\BCbot \to V$ be a continuous map, where  $V$ is a separable Banach space of type $2$. 
Then $\norm{T(\beta_k(X))}$ has finite $n$-th moments for all $n\geq 0$.
\end{custom}

Once the existence of barcode expectations is settled, it is important to know how to calculate them for random point clouds of bigger and bigger samples, drawn from an unknown distribution.
The TDA pipeline is too complicated for permitting one to find an explicit symbolic way for such calculations in general. 
The only reasonable way of doing so is to make an \emph{educated guess}! 
We infer directly from Theorem \ref{theorem almost sure limit of number of points}

\begin{corollary}
Let $M$ be a metric space, and $X_1,X_2, \ldots$ an i.i.d.~sequence of $M$-valued random variables.
Set $k\in \N_0$, and put $P_n=\{X_1,X_2, \ldots,X_n\}$.
If the distribution of the $X_i$ has support equal to a compact subset $C \subset M$, and if $T:\BCbot \to V$ is a continuous map to a Banach space of type 2, then
\[
T\left(\beta_k(C)\right) = \lim_{n\to \infty} \E[T\left(\beta_k(P_n)\right)].
\]
\end{corollary}

For some specific underlying probability distributions, explicit calculations and more careful asymptotic estimates may be possible.
We consider the simplest (and paradigm) example of the circle $\S^1 \subset \R^2$, and i.i.d.~points sampled from it.
The interesting barcode here is the $\beta_1$-barcode, and it is uniquely determined by its length. 
In Theorem \ref{theorem expectation circle} we give an explicit formula for the expectation of the length.

The principal contribution of this work is that we devise 
  a new concrete general framework for analysis of random finite point clouds and their corresponding barcodes. 
The fact that the proofs of our main results are not technically involved is in our opinion a firm indication that the framework here proposed 
is natural and potentially very useful in studying more sophisticated TDA questions. 

\subsection*{Related work}
There are a number of related approaches to studying the statistical properties of persistent homological estimators (see~\cite{2017arXiv171004019C} for an overview). 

A closely related work is by Bubenik~\cite{Bub15}, who develops statistical inference via an embedding with ``persistence landscapes'', which is further studied in~\cite{ChazalFasy2015} and~\cite{pmlr-v37-chazal15}. Like Bubenik, we use CLT and LLN theory in Banach spaces, but on an object different from his. Unlike him, we study  natural geometric and probabilistic limits directly on the barcodes of large point clouds (Theorem~\ref{theorem almost sure limit of number of points} and Theorem~\ref{theorem consequence of NSW}).  
In particular, in Theorem~\ref{theorem consequence of NSW}, we establish a connection with the work of \cite{NSW08}. 
The just mentioned theorems are linked in spirit
to~\cite{Bobrowski2015}, who also study homology approximations based on large point clouds
drawn from a compact manifold, and their analysis is based on \cite{NSW08} as well.
Unlike us, for large $n$, Bobrowski and Mukherjee~\cite{Bobrowski2015} approximate simultaneously the homology of the manifold in a large range $m_n$ of degrees, with the homology in the corresponding degrees of the point cloud inflated by $r_n$, where $r_n$
is a power of $\frac{\log(n)}{n}$. 
In our context of persistent homology, we look at a continuum (a segment) of radii (away from zero) and aim to match, for large $n$, the homotopy type of the manifold with that of the inflated point cloud, simultaneously for all radii in thus fixed interval. 

Chazal et al.~\cite{Chazal:2015:CRP:2789272.2912112} establish convergence rates for metric spaces endowed with a probability measure that satisfies the $(a, b)$-standard assumption, see section 2.2 of that paper. In our study of almost sure convergence, we do not impose any conditions on the measure (except for compact support),  see Theorem \ref{theorem almost sure limit of number of points}. Hiraoka, Shirai, and Trinh~\cite{Hiraoka2018}, Owada~\cite{owada2018}, Adler 
and Owada~\cite{owada2017} also study limit theorems for persistence diagrams; but in their case, the point clouds are stationary point processes on $\RR^n$. Similar results also appear in~\cite{Chazal:2015:CRP:2789272.2912112}.

The foundational work of Mileyko, Mukherjee and Harer in~\cite{0266-5611-27-12-124007} introduces probability measures on barcode space, and these ideas are developed further (with Turner) in the context of Frech\'et means as ways of summarizing barcode distributions in~\cite{Turner2014}. Since we work with embeddings into Banach spaces, we do not need to rely on the theory developed in these papers.

Another active area of research in TDA deals with topological features of random simplicial complexes and noise~\cite{Bobrowski2018, Adler2014}. The present paper has a different focus, but it would be interesting to incorporate noise into our framework. This will be the subject of a forthcoming work.

\subsection*{Outline of the article}
In Section \ref{section persistent homology and barcodes}, we recall the definition of persistent homology, barcodes, and the space of barcode representations. This is the basis for what follows. Most results presented in this section are not new, nevertheless we occasionally included arguments (such as for Lemma \ref{lemma hausdorff metric}) to make the text more self-contained. 
As explained above, equivalent statements to Proposition \ref{proposition barcodes for compact metric spaces}, where the definition of barcode representations is extended from finite point clouds to compact subsets of a given metric space (with the induced metric), already appeared in the literature. As our approach through completions is conceptually very clear, we nevertheless chose to include it. 
This definition is fundamental for the rest of the paper. The generalized barcode representations live in the completion $\BCbot$ of the space of barcodes with respect to the bottleneck distance, thus they can be thought of as representing barcodes consisting of countably many intervals with a finite metric distance to a given barcode.
 In the same spirit as Proposition \ref{proposition barcodes for compact metric spaces} we show in Theorem \ref{theorem tame is decomposable} that the filtration associated to a limit of tame functions also has an associated barcode representation.
 
Following Bubenik \cite[Section 3.2]{Bub15}, we study a Law of Large Numbers and a Central Limit Theorem in Section \ref{section clt}. The main new contribution here is Theorem \ref{theorem hypothesis of lln and clt for compact} as explained above. Section \ref{section true persistent homology} contains the probabilistic limit theorem \ref{theorem almost sure limit of number of points} for barcodes which are probably the most fundamental contribution of this article. It heavily relies on Lemma \ref{lemma almost sure limit of number of points} which is a more geometric limit theorem for random point clouds in the space of compact subspaces of $\R^d$.

Section \ref{section sphere} is independent of the preceding two sections and gives a hint at a more quantitative version of a limit theorem. For this we consider the simplest nontrivial example of a compact metric space in $\R^2$ -- the circle $\S^1$. We fix the number $n$ of points and we would like to determine the expected barcode of a random $n$-point cloud (consisting of independent uniform samples from $\S^1$). To give meaning to this idea, we need to find some quantity which determines the barcode and over which we can average (in order to speak of expectations). In this simple case, the elementary geometry of the circle (and of point clouds on it) only allow for a restricted barcode which is entirely determined by its length, see Corollary \ref{corollary bn leq one} and the discussion thereafter. We give explicit formulas for the \emph{expected length} for $n=3$ in Proposition \ref{theorem expectation circle 3points}  and arbitrary $n$ in Theorem \ref{theorem expectation circle}.

\enlargethispage{\baselineskip}

Such quantities as the length which determine the barcode completely can of course no longer be explicitly given for arbitrary compact submanifolds of $\R^d$ which is why in order to talk about expectation we consider embeddings (or more generally continuous maps) $T:\BCbot\to V$ to some Banach space $\left(V,\norm\cdot_V\right)$. Building on the work of Niyogi--Smale--Weinberger \cite{NSW08}, who investigated when an $\veps$-neighborhood of a random point sample on a compact submanifold of $\R^d$ is homotopy equivalent to that manifold, we give an estimate for the distance in $V$ of the expectation of the transform (under $T$) of the  barcode for a random $n$-point cloud for fixed $n$ from  the transform of the barcode for the manifold $M$ from which the point cloud is sampled. 

Section \ref{section embedding barcodes} shows that our hypotheses on the existence of a (Lipschitz continuous) map from the barcode space to a Banach space can be fulfilled using functions introduced in Kalisnik's work~\cite{Kal18}. Finally, Section \ref{section discussion} gives a glimpse at open problems in this context.

\subsection*{Notation and Conventions} Let $(M,d)$ be a metric space. 
For $x \in M$ and ${t\in \R_{\geq 0}}$, let $B_t(x)=\{y \in M \mid d(x,y)<t\}$ be the open $t$-ball of $x$ and $\overline{B}_t(x)=\{y \in M \mid d(x,y)\leq t\}$
 the closed $t$-ball around $x$. For a subset $P\subset M$ we will denote $P_t:=\{x\in M\mid d(x,P)\leq t\}$ the $t$-neighborhood of $P$ which is closed if $P$ is.
 
 We denote by $\sP(X)$ the power set of $X$ and by $F(X) \subset \sP(X)$ the set of finite nonempty subsets of $X$.
 Throughout this paper, we take homology groups with coefficients in a field $\mathbf k$. For $n\in\N$ denote by $[n]$ the set $\{1,\ldots,n\}$.
 
 Recall that a \emph{multiset} is a set $A$ together with multiplicities, i.e., a map $A\to \N_0$. We will usually suppress the map in the notation and just speak of a multiset $A$. Also, we will use set notation such as $A=\{x_1,x_2,x_3,\ldots\}$.
 
We use $\Theta$ for asymptotically comparable in the Big O notation. For example, $f=\Theta (\frac{\log(n)}{n})$ if and only if there exist positive constants $C_1$ and $C_2$ such that $C_1\frac{\log(n)}{n}\leq f(n) \leq C_2\frac{\log(n)}{n}$ for all large $n$.

\subsection*{Acknowledgements}

We thank Bernd Sturmfels for his interest and the MPI Leipzig where a large part of the article was written for the hospitality and excellent working conditions. Christian Lehn benefited from discussions with Joscha Diehl and Mateusz Micha\l{}ek. Vlada Limic benefited from discussions with Vitalii Konarovskyi.

Christian Lehn was supported by the DFG through the research grants Le 3093/2-2 and  Le 3093/3-1.
Vlada Limic was supported by the Friedrich Wilhelm Bessel Research Award from the Alexander von Humboldt Foundation.

\section{From persistent homology to barcodes}\label{section persistent homology and barcodes}

\subsection{Persistence}\label{section persistent homology}
In many applications data lies in a metric space, for example, in a subset of a Euclidean space with an inherited distance function.
From this (necessarily finite, and often large) sample one wishes to learn some basic characteristics, such as the number of components or the existence of holes and voids, of the underlying space from which we sampled. 
Finite metric spaces are discrete spaces, and as such do not per se have interesting topological structure in their own right. 
The philosophy of topological data analysis is that data does have an inherent topology and in order to 
uncover it, one assigns a 1-parameter family of topological spaces or a filtration to a finite metric space $M$~\cite{topodata, carlsson2014, elz-tps-02}. Applying the degree-$k$ homology functor $H_k$ to this filtration yields what is called a persistence module 
\cite{chazal:hal-01330678}.
\begin{definition}\label{definition persistence module}
Let $\k$ be a field. 
A {\bf persistence module} (over $\k$) is an indexed family of vector spaces 
\[
V=\left(\{V_t\}_{t\in \R},\{\phi_s^t\}_{s \leq t \in \R}\right)
\] 
of $\mathbf k$-vector spaces $V_t$ and linear maps $\phi_s^t:V_s \to V_t$ for every $s \leq t$ such that $\phi_t^t=\id_{V_t}$ and $\phi_r^t = \phi_s^t \circ \phi_r^s$ for all
 $r \leq s \leq t$.
\end{definition}

One could also replace the field $\k$ by a ring $R$ (e.g. $R=\Z$ is a natural choice) and define an $R$-persistence module by replacing $\k$-vector spaces by $R$-modules in the above definition. This might give finer information about the topology of the point clouds, but is also much more complicated from the representation theoretic point of view, see e.g. the discussion in \cite{topodata} before Theorem 2.10 (p. 267). For example, analogs of essential results like Gabriel's theorem (stated here as Theorem \ref{theorem tame is decomposable}) are not available for $R=\Z$. As our work builds on that in an essential way, we work with fields and vector spaces throughout the paper. 

Recall that if we work with field coefficients, homology is a collection of functors $(H_n)_{n\in \N_0}$ from the category of topological spaces to the category of $\k$-vector spaces. We refer the reader to standard textbooks such as 
Bredon \cite{Bre97} or Hatcher \cite{Hat02}. It is sometimes useful to consider \emph{reduced homology} whose definition we briefly recall: denote by $\pt$ the one point space. Then for every topological space $X$ there is a unique continuous map $p_X:X \to \pt$. One defines the reduced degree $k$ homology of $X$ as 
\[
\tilde H_k(X):= \ker\left( H_k(X) \to H_k(\pt)\right)
\]
where $H_k(X) \to H_k(\pt)$ is the map in homology induced by $p_X$. As $H_k(\pt)=\k$ if $k=0$ and is trivial otherwise, we have $H_k(X)=\tilde H_k(X)$ for every $k \neq 0$. Reduced homology is also a functor on the category of topological spaces.

\begin{definition}\label{definition sublevelset filtration}
Let $X$ be a topological space and let $f\colon X \to \R$ be a continuous function. 
This function defines a filtration, called the {\bf sublevelset filtration} of $(X, f)$, by setting ${X_{t}=f^{-1}\left((-\infty, t]\right)}$. 
For $k \in \N_0$ the sublevel set filtration of $(X, f)$ defines a persistence module $(\PH_k(X,f),\phi)$ by ${\PH_k(X,f)_t= \tilde H_k(X_{t})}$ and 
$\phi_s^t:\tilde H_k(X_{s}) \to \tilde H_k(X_{t})$ induced by the inclusion $X_{s} \xhookrightarrow{} X_{t}$. 
For $X\subset \R^d$ we will simply write $\PH_k(X)$ instead of $\PH_k(X,f)$ if $X \subset \R^d$ and ${f:\R^d \to \R}$ is the distance--to--$X$ function. We refer to $\PH_k(X)$ (respectively $\PH_k(X,f)$) as the {\bf persistent homology} in degree $k$ of $X$ (respectively of $(X,f)$).
\end{definition}

\begin{definition}\label{definition tame}
A persistence module $V$ is called {\bf tame} if all $V_t$ have finite dimension and there exist finitely many $t_1 < \ldots < t_m \in \R$ such that $\phi_s^t$ is an isomorphism whenever $s,t \in (t_i,t_{i+1})$ for some $i$ (where we set $t_0=-\infty$, $t_{m+1}=\infty$). The function $f$ is called {\bf tame} if the module $\PH_k(X,f)$ is tame for all $k$.
\end{definition}

\begin{example}\label{example non tame}
It is clear that for an arbitrary smooth manifold $M \subset \R^d$ the $k$-th persistence module $\PH_k(M)$ is not necessarily tame. Take for example a strictly decreasing sequence $(r_n)_{n \in \N}$ of positive rational numbers such that $\sum_{n \in \N} r_n < \infty$ and put $R_n:=\sum_{m=1}^n r_n$. If $M$ is the union over all ${n \in \N}$ of circles $K_n$ with radius $R_n$ centered at the origin, then the persistent homology $\PH_1(M)$ will decompose as a direct sum of interval modules and this decomposition will give rise to an element $b\in \BCbot\ohne \BC$.
\end{example}

In certain cases a persistence module can be expressed as a direct sum of ``interval modules'', which can be thought of as the building blocks of the theory. 
Here we have four types of intervals and recall the representation from~\cite{chazal:hal-01330678}:
\begin{center}
\begin{tabular}{ccc}
\footnotesize{interval}
&
\footnotesize{decorated pair}
\\
$(p,q)$ & $(p^+,q^-)$\\
$(p,q]$ & $(p^+,q^+)$\\
$[p,q)$ & $(p^-,q^-)$\\
$[p,q]$ & $(p^-,q^+)$\\
\end{tabular}
\end{center}

\begin{definition}\label{definition interval module}
For an interval $(p^*, q^*)$, where $^*$ is either $+$ or $-$, denote by $\mathbb{I}(p^*, q^*)$ the persistence module 
\[
 (\mathbb{I}(p^*, q^*))_t=\begin{cases}{\mathbf k}, & \textrm{ for } t \in (p^*, q^*)\\ 0, & \textrm{ otherwise}\end{cases} \textrm{ and }
 \phi_s^t = \begin{cases}\id_{{\mathbf k}}, & \textrm{ for } s \leq t, \textrm{ and }s, t \in (p^*, q^*)\\ 0, & \textrm{ otherwise}\end{cases}.
\]
\end{definition}
\begin{definition}\label{definition persistent homology}
A persistence module $V$  over ${\mathbf k}$ is called \emph{decomposable} if it can be decomposed as a direct sum
 \[
V \isom \bigoplus_{m\in \Lambda} \mathbb{I}(p_m^*, q_m^*),
 \]
where $\Lambda$ is some index set and $*\in\{+,-\}$. If $V$ is decomposable, then the {\bf barcode} associated to $V$ is the {\bf multiset}
\[
\left\{(p_m^*, q_m^*) \mid m\in \Lambda  \right\}.
\]
We call a decomposable persistence module $V$ \emph{of finite type} if $\Lambda$ is a finite set.
\end{definition}

\begin{remark}
The barcode of $V$ is also called the {\bf persistence} of $V$.
\end{remark}
Not all persistence modules decompose in this way~\cite{chazal:hal-01330678}, and there is a considerable body of literature trying to ascertain under 
which conditions persistence modules are decomposable~\cite{quiver, Chazal:2009:PPM:1542362.1542407, chazal:hal-01330678, CB15}. 
We will restrict to the case of most interest to us.
\begin{theorem}[Gabriel \cite{quiver}]\label{theorem tame is decomposable} Let $X$ be a topological space and let $f:X\to \R$ be a tame 
function in the sense of Definition~\ref{definition tame}. Then $\PH_k(X,f)$ is decomposable and of finite type.
\end{theorem}
\begin{example}\label{example tame}
Examples of $(X, f)$ with a tame function $f$ include:
\begin{itemize}
\item
$X$ a compact manifold and $f$ a Morse function (where tameness is the result of Morse theory, see \cite{Mil63} for a general reference to this classical field).
\item
$X$ a compact polyhedron and $f$ a piecewise linear function, see Theorem 2.2 in \cite{chazal:hal-01330678}.
\item $X=\R^d$ and $f$ the distance to $P$ function for a finite set $P \subset \R^d$. In this case, tameness is a direct consequence of the nerve theorem. A textbook reference for the general nerve theorem is e.g. \cite[Corollary 4G.3]{Hat02}.
\end{itemize}
\end{example}

Let $P \subset \R^d$ be a finite set and ${f:\R^d \to \R}$ the distance to $P$ function. Then $P_t = f^{-1}((-\infty,t])$ is just the closed $t$-neighborhood of $P$ and $\PH_k(P)_t=\tilde H_k(P_{t})$ is decomposable by Theorem \ref{theorem tame is decomposable} for $k \in \N_0$. Furthermore, all non-zero intervals appearing in the barcode are 
closed on the left and open on the right (also known as \emph{closed-open type}), or equivalently of the third 
type in the above table describing the decorated pair notation. 

We can of course define $\PH_k(P)$ even when $P$ is not finite as it may still be decomposable. 
For example, $\PH_k(P)$ is decomposable and of finite type for a semi-algebraic set $P$ as a consequence of Hardt's theorem, see the discussion in section 3.2 of~\cite{HW18}.

\subsection{Persistent Homology of Finite Subsets of Metric Spaces}\label{section metric spaces}

As mentioned in Example \ref{example tame}, the persistent homology $\PH_k(P)$ of a finite point cloud $P \subset \R^d$ can be calculated using the nerve theorem. It tells us in particular, that the homology of $P_t$ is the same as the homology of a simplicial complex, the so-called \emph{\v Cech complex} with parameter $t$. 

Recall that given a metric space $M$, a finite set $P\subset M$, and a parameter $t \geq 0$, the {\v{C}ech complex} $\check{C}_t(P)$ is the abstract simplicial complex whose vertex set is $P$, and where $\{x_0, x_1, \ldots, x_k\}$ spans a $k$-simplex if and only if $\bigcap_{i=0}^k\overline{B}_t(x_i) \neq \emptyset$. 
The \emph{\v{C}ech filtration} of $P$ is the nested family of \v{C}ech complexes obtained by varying parameter $t$ from $0$ to $\infty$. This can be used as a definition.

\begin{definition}\label{definition persistent homology for finite metric spaces}
Let $M$ be a metric space and let $P\subset M$ be a finite subset. For $k\in \N_0$ we define the persistent homology $\PH_k(P)$ in degree $k$ of $P$ to be the persistence module obtained from taking the homology of the nested family of \v{C}ech complexes associated to $P$. In formulas:
\[
\PH_k(P)_t:= \tilde H_k(\check{C}_t(P)) \quad \textrm{for } t\in\Rplus.
\]
\end{definition}

From the construction and Theorem \ref{theorem tame is decomposable}, we immediately deduce

\begin{corollary}\label{corollary persistent homology for finite metric spaces}
Let $M$ be a metric space and let $P\subset M$ be a finite subset. Then for every $k\in \N_0$, the persistence module $\PH_k(P)$ is tame and decomposable.\qed
\end{corollary}

Note that the two definitions (Definition \ref{definition sublevelset filtration} and Definition \ref{definition persistent homology for finite metric spaces}) of $\PH_k(P)$ for a finite subset $P \subset \R^d$ coincide.

\subsection{Barcode Space}\label{section barcodes}

In this subsection we describe a useful way of representing barcodes. 
Given an interval $I\subset \R_{\geq 0}$ of finite length, we encode it as a point 
$(x, d)\in\R_{\geq 0}^2$ where $x$ is the left endpoint of $I$ and $d$ is its length. 
The price we pay with this simplified representation is the loss of information about the 
inclusion of endpoints of the intervals. However, restricted to only one 
single interval type, this representation map is injective.
In the cases we are mainly interested in, this is indeed the case. We are led to the following
 
\begin{definition}\label{definition barcode}
Let us denote $A:= \coprod_{n \in \N_0} \RR_{\geq 0}^{2n}$. Let $\sim$ on $A$ be an equivalence relation generated by the relations
\begin{gather*}
(x_1,d_1,\ldots,x_n,d_n) \sim (y_1,e_1,\ldots,y_m,e_m) \Longleftrightarrow\\ (x_{\sigma(1)},d_{\sigma(1)},\ldots,x_{\sigma(n)},d_{\sigma(n)}) = (y_1,e_1,\ldots,y_n,e_n) \textrm{ and } \\ e_{n+1}=\ldots = e_m=0 \textrm{ for some } \sigma \in S_n, \ n\leq m \in \N_0
\end{gather*}
where $S_n$ denotes the symmetric group on $n$ elements. 
A {\bf barcode representation} is an equivalence class of $(x_1,d_1,\ldots,x_n,d_n)$ with respect to $\sim$. 
 The {\bf space of barcode representations} is the quotient of the disjoint union $A$ by the equivalence relation defined above:
\[
\BC:= \left. A \middle/  \sim \right..
\]
For simplicity, we will sometimes also  refer to $\BC$  as the {\bf barcode space}. We denote by $\BC_n \subset \BC$ the image of $\coprod_{m \leq n} \RR_{\geq 0}^{2m}$ under the canonical map $A \to \BC$.

We adopt the notation of Definition \ref{definition interval module}. Let $b=\{ \left(x_1^*, (x_1+d_1)^*\right), \ldots, \left(x_n^*, (x_n+d_n)^*\right)\}$ with $*\in\{+,-\}$ be a barcode such that all intervals have non-negative left endpoint $x_i$ and finite length $d_i$. Then we call $(x_1,d_1,\ldots,x_n,d_n) \in \BC$ the \textbf{barcode representation of the barcode $b$}.
\end{definition}

The equivalence relation $\sim$ defined above says that two barcode representations are equivalent if they coincide up to permutation of intervals and after deleting zero length intervals (i.e.\ $(x_i,d_i)$ with $d_i=0$). 

As already pointed out, given a finite subset $P$ of a metric space $M$, the persistence module $\PH_k(P)$ is decomposable and of finite type by Theorem \ref{theorem tame is decomposable}, Example \ref{example tame}, and Corollary \ref{corollary persistent homology for finite metric spaces}. 
Therefore, there is an associated barcode all of whose intervals have finite length. This -- and in fact only for $k=0$ -- is where we need to use reduced instead of ordinary homology. We can define the following barcode map from the set of finite nonempty subsets of a metric space $M$ to the barcode space.

\begin{definition}\label{definition barcode map}
Let us fix $k\in\N_0$. Given a finite subset $P$ of some metric space $M$, we denote by $\beta_k(P)$ the barcode representation of the barcode associated to the persistence module $\PH_k(P)$. This defines a map
\[
\beta_k : F(M) \to \BC
\]
where $F(M)$ is the set of finite nonempty subsets of $M$. We will refer to this map as the $k$-th barcode map.
\end{definition}

The barcode space comes equipped with natural metrics. In order to define them, we first specify the distance between any pair of intervals, as well as the distance between any interval and the equivalence class of the zero length interval which for this purpose is represented by the set $\Delta = \{(x, 0)\,|\, -\infty < x < \infty \}$.  We put
\[
\textrm{d}_\infty \left((x_1, d_1), (x_2, d_2)\right) = \max \left(|x_1-x_2|, |(x_1 +d_1) - (x_2+d_2)|\right).
\]
The distance between (the representation of) an interval and the set $\Delta$ is
\[
\textrm{d}_\infty ((x, d), \Delta) = \frac{d}{2}.
\]
Recall that $[n]=\{1,2,\ldots,n\}$.
 Let $b_1 = \{I_i\}_{i \in [n]}$ and $b_2 = \{J_j\}_{j \in [m]}$ be barcodes. For any bijection $\theta$ from a subset $A \subseteq [n]$ to $B \subseteq [m]$, the \emph{penalty} $P_\infty(\theta)$ of $\theta$ is
\begin{equation}\label{eq penalty bottleneck}
P_\infty (\theta) = \max\left(\max_{a\in A}\left(\textrm{d}_\infty(I_a, J_{\theta(a)})\right), \max_{a\in [n] \setminus A} \textrm{d}_\infty (I_a, \Delta), \max_{b\in [m]\setminus B} \textrm{d}_\infty (I_b, \Delta)\right).
\end{equation}
\begin{definition}
The \emph{bottleneck distance} is defined by
\[
\textrm{d}_\infty(b_1, b_2) = \min_\theta P_\infty(\theta),
\]
where with the notation above the minimum is over all possible bijections $\theta$ from subsets $A\subset [n]$ to subsets $B\subset [m]$. 
\end{definition}
There are other metrics also commonly used  for barcode spaces. Keeping the notation and changing the penalty \eqref{eq penalty bottleneck} for the bottleneck distance to
\begin{equation}\label{eq penalty wasserstein}
P_p(\theta) = \sum_{a\in A}\textrm{d}_\infty(I_a, J_{\theta(a)})^p +\sum_{a\in [n] \setminus A} \textrm{d}_\infty (I_a, \Delta)^p +\sum_{b\in [m]\setminus B} \textrm{d}_\infty (I_b, \Delta)^p
\end{equation}
yields the \emph{$p$th-Wasserstein distance} ($p\geq 1$) between $b_1, b_2 \in \BC$:
\[
\textrm{d}_p(b_1, b_2) = \left(\min_\theta P_p(\theta)\right)^{\frac{1}{p}}.
\]

Let us consider an example in order to get acquainted with these metrics.

\begin{example}\label{example bottleneck}
Let $\BC_1 \subset \BC$ consist of barcodes containing a single interval (bar). We set $b_1=(x_1, d_1), b_2=(x_2, d_2) \in \BC_1$ and calculate
\[
d_\infty(b_1,b_2)=\min\left(\max\left(|x_1-x_2|, |x_1+d_1-(x_2+d_2)|\right), \max\left(\frac{d_1}{2},\frac{d_2}{2}\right)\right).
\]
Then we see that if for arbitrary fixed $x_1,x_2 \in \Rplus$ the length of both intervals is small, the bottleneck distance between $b_1$ and $b_2$ is equally small, even if the intervals are far away from each other. 
The $p$th-Wasserstein distance behaves similarly. 
\end{example}

The barcode space $\BC$ is not a complete metric space, neither with respect to the bottleneck, 
nor with respect to any of the Wasserstein distances~\cite{0266-5611-27-12-124007}.
 This is a consequence of the fact that appending bars of smaller and smaller but nonzero length to any given barcode can easily yield a Cauchy 
 sequence of barcodes, with respect to any of
 the above metrics, and clearly without a limit in $\BC$.  
For the sake of concreteness, let $x >0$ be fixed, and consider the barcode $b_n$ consisting of all intervals $I_k:=(x,\frac{1}{k})$ for all $1\leq k\leq n$ (so that $b_n \in B_n$).
In this case, we have for $n<m$
\[
d(b_n,b_m) \leq \max_{n+1\leq k \leq m} d_\infty(\{I_k \},\Delta) = \frac{1}{2(n+1)}.
\]
A limit could only be a barcode consisting of infinitely many bars, which is impossible. 

In order to overcome this problem, we shall consider the completions 
\begin{equation}\label{eq completions}
\left(\BCwas,d_p\right) \quad  \textrm{and} \quad \left(\BCbot,d_\infty\right)
\end{equation}
of $\BC$ with respect to the Wasserstein and bottleneck distances.

\subsection{Limits of Barcodes}\label{section true persistent homology}

In subsection~\ref{section persistent homology} we recalled the classical construction of barcodes from finite point clouds. 
Here we present a generalization, which is natural in the context of our probabilistic investigations.
Let $(M,d)$ be a metric space and consider the family 
\[
K(M):=\left\{ Y \subset M \mid Y \textrm{ compact, non-empty}\right\}
\]
 of all compact subsets of $M$. 
Together with the Hausdorff distance
\begin{equation}\label{eq hausdorff metric}
d_H(A,B):= \max\left(\inf\left\{t\in\Rplus\mid A \subset B_t\right\},\inf\left\{t\in\Rplus\mid B \subset A_t\right\}\right),
\end{equation}
the set $K(M)$ becomes a metric space. 
It is well known that $(K(M),d_H)$ is complete whenever $(M,d)$ is complete, and compact whenever $(M,d)$ is compact. 
Given a bounded subset $A \subset M$, we consider the continuous function, the ``distance from $A$'', defined by
\[
d_A:M\to \Rplus, \quad d_A(x):=\inf\{d(x,y) \mid y \in A \}.
\]

We can also describe compact metric spaces in terms of functions. The following result should be rather standard, but it turned out to be easier to give a proof than to find an exact reference.
\begin{lemma}\label{lemma hausdorff metric}
Let $M$ be a metric space, and denote by $\left(L_\infty(M),\norm{\cdot}_\infty\right)$ the Banach space of bounded functions $f:M \to \R$, equipped with the 
supremum norm.
\begin{enumerate}
	\item\label{lemma hausdorff metric item one} For $A,B \in K(M)$ the function $d_A-d_B$ is bounded on $M$.
	\item\label{lemma hausdorff metric item two} 
The function $n_\infty:K(M) \times K(M) \to \R_{\geq 0}$, $n_\infty(A,B):=\norm{d_A-d_B}_\infty$ defines a metric on $K(M)$
such that $$(K(M),d_H) \to (K(M),n_\infty), \quad A\mapsto A $$ is an isometry.
	\item\label{lemma hausdorff metric item three} If $M$ is compact, then the function $d_A$ for $A\subset M$ is bounded and $A\mapsto d_A$ defines a continuous injective map 
$$(K(M),d_H)\into \left(L_\infty(M),\norm{\cdot}_\infty\right),$$ 
which is an isometry of metric spaces onto its image.
\end{enumerate}
\end{lemma}
\begin{proof}
For \eqref{lemma hausdorff metric item one} let us denote $R:=\sup_{a\in A, b\in B} d(a,b)$ which is $< \infty$ by compactness. For a given $x\in M$, we choose $a\in A, b \in B$ such that $d_A(x)=d(a,x)$, $d_B(x)=d(b,x)$ which is again possible by compactness.
Without loss of generality $d_A(x) \geq d_B(x)$. The triangle inequality gives
\[
\abs{d_A(x) - d_B(x)} = d(a,x)-d(b,x) \leq d(a,b) 
 \leq R
\]
and the claim follows.

For \eqref{lemma hausdorff metric item two} let $A,B \in K(M)$. We will first prove that $d_H(A,B) \leq \norm{d_A-d_B}_\infty$. 
Suppose that $\abs{d_A(x)-d_B(x)} \leq t$ for some $t\in \Rplus$ and for all $x \in M$. 
Then in particular for $a \in A$ we deduce $d_B(a) \leq t$ so that $A \subset B_t$. 
By symmetry the other inclusion follows and therefore $d_H(A,B) \leq t$.

For the inequality in the other direction, let us now assume that for some $t\in \Rplus$ we find $A \subset B_t$ and $B\subset A_t$. Let $x \in M$ be given. It suffices to show that $\abs{d_A(x)-d_B(x)} \leq t$. We may assume $d_A(x)-d_B(x) > 0$. By compactness, the infimum is a minimum so that $d_A(x) = d(a,x)$ and $d_B(x) = d(b,x)$ for some $a \in A$, $b \in B$. As $B \subset A_t$ there is $a' \in A$ such that $d(a',b) \leq t$. From $d_A(x) = d(a,x)$ it follows that $d(a',x) \geq d(a,x)$ and we infer
\[
\abs{d_A(x)-d_B(x)} = d(a,x)-d(b,x) \leq d(a',x) -d(b,x) \leq d(a',b) \leq t.
\]
Thus $\norm{d_A-d_B}_\infty \leq t$.

Let us now prove \eqref{lemma hausdorff metric item three}. 
Every compact metric space has a finite radius $R:=\sup_{x,y\in M} d(x,y)$. Obviously $\norm{d_A}_\infty \leq R$. 
The rest of the claim follows from \eqref{lemma hausdorff metric item two}.
\end{proof}

\begin{proposition}\label{proposition barcodes for compact metric spaces} Let $k \in \N_0$ be a nonnegative integer and $M$ be a metric space.
\begin{enumerate}
	\item\label{proposition barcodes for compact metric spaces item one} 
The map $\beta_k: F(M) \to \BCbot$ is Lipschitz continuous with Lipschitz constant equal to $1$.
	\item\label{proposition barcodes for compact metric spaces item two} 
There is a unique continuous extension $K(M) \to \BCbot$ of $\beta_k:F(M) \to \BC \subset \BCbot$. 
We will denote it by the same symbol $\beta_k$.
The extended map is also Lipschitz continuous with Lipschitz constant $1$.
\end{enumerate}
\end{proposition}
\begin{proof}
The claim in (1) was proved in~\cite{Chazal:2009:GSS:1735603.1735622}.

For (2), we first show that $F(M) \subset K(M)$ is dense. Given a compact subset $K\subset M$ and $\veps > 0$ we will show that $B_\veps(K):=\{A \in K(M)\mid d_H(A,K)<\veps\} \subset K(M)$ intersects $F(M)$ nontrivially. Since $K$ is compact, there is $P=\{x_1,\ldots,x_n\}\subset  K$ such that $K \subset B_\veps(P)$. On the other hand, $P \subset K \subset K_\veps$ so that $d_H(P,K) < \veps$. As Lipschitz functions are in particular uniformly continuous, $\beta_k:F(M) \to \BCbot$ extends to $K(M)$. The fact that the extension is again Lipschitz with the same Lipschitz constant is also standard.
\end{proof}
\begin{definition}
We call the map $\beta_k:K(M)\to \BCbot$ the barcode map and for a compact set $K \subset M$ we refer to $\beta_k(K)$ as the $k$-th barcode of $K$.
\end{definition}

This definition and the definition of \cite{Chazal:2009:GSS:1735603.1735622,chazal:hal-01330678} are equivalent concepts as both produce barcode maps which are continuous functions from the space $K(M)$ to some space of barcodes and they coincide on the dense subset $F(M) \subset K(M)$ of finite subsets of $M$.

\begin{remark}\label{remark barcodes for totally bounded spaces}
Note that the barcode map $\beta_k$ can easily be extended to a map on totally bounded sets. Since in the proof of Proposition \ref{proposition barcodes for compact metric spaces} we reduced to the case where $M$ is complete, then a totally bounded subset is  compact if and only if it is closed. Therefore, for every totally bounded set there is a compact set (its closure) at Hausdorff distance zero (see \eqref{eq hausdorff metric}, although totally bounded spaces only form a pseudo metric space for the Hausdorff spaces). One way to define a barcode for a totally bounded set is therefore to define it via Proposition \ref{proposition barcodes for compact metric spaces}  as the barcode of its completion which is compact.
\end{remark}
One can naturally generalize Proposition \ref{proposition barcodes for compact metric spaces} to the setting of tame functions. 
\begin{definition}
Let $M$ be a metric space. We denote by $C(M,\R)$ the set of continuous functions with values in $\R$ and endow it with the metric
\[
d:C(M,\R) \times C(M,\R) \to[] [0,\infty], \quad d(f,g)= \norm{f-g}_\infty.
\]
We will denote by $T(M) \subset C(M,\R)$ the subset of tame functions and by $\wh T(M)$ its completion.
\end{definition}
There is no harm in allowing the metric to take value  $\infty$. 
The induced topology is the same as the one induced by the metric
\[
(f,g) \mapsto d'(f,g):=\frac{d(f,g)}{1+d(f,g)} \in [0,1].
\]
The metrics $d,d'$ also feature the same notion of Cauchy sequences. Working with $d$ is however more appropriate for the inequalities we need.
\begin{theorem}\label{theorem barcodes for tame functions} Let $k \in \N_0$ be a nonnegative integer and let $M$ be a metric space.
\begin{enumerate}
	\item\label{proposition barcodes for tame functions item one} The map $\beta_k: T(M) \to \BC$ is Lipschitz continuous with Lipschitz constant equal to $1$.
	\item\label{proposition barcodes for tame functions item two} There is a unique continuous extension $\wh T(M) \to \BCbot$ of $\beta_k:T(M) \to \BC \subset \BCbot$. 
As before we denote it by the same letter, and note that the extension is also $1$-Lipschitz continuous.
\end{enumerate}
\end{theorem}
\begin{proof}
As in Proposition \ref{proposition barcodes for compact metric spaces}, the first part follows from~\cite{Chazal:2009:GSS:1735603.1735622}. The second part is implied by the same extension argument for uniformly continuous maps. Note that Lipschitz continuity implies uniform continuity.
\end{proof}

\section{Barcodes of compact sets as almost sure limits}\label{section approximation two}

In this section, we will address a very natural convergence problem for stochastic barcodes. It is somewhat surprising that this question has never been addressed before, at least not in full generality. 

Let $M$ be a metric space. We consider i.i.d.~$M$-valued random variables $X_1,X_2, \ldots$ whose distribution has support equal to a compact subset $C \subset M$. Recall that the \emph{support} of a measure $\mu$ on a $\sigma$-algebra containing the Borel $\sigma$-algebra $B(M)$ is defined to be the closed subset
\[
\supp(\mu):= \{ x \in M \mid \forall \veps >0: \mu(B_\veps(x))> 0\}.
\]
Let us consider the finite random set $P_n=\{X_1,X_2, \ldots,X_n\}$ and for a fixed $k$ the sequence of barcodes $\left(\beta_k(P_n)\right)_{n \in \N}$. We would like to describe the limit of this sequence for $n\to \infty$. If $P_n$ were a deterministic sequence approaching $C$ in the Hausdorff distance, then the limit of $\beta_k(P_n)$ would be $\beta_k(C)$ by definition of the latter, see Proposition \ref{proposition barcodes for compact metric spaces}. Now, the $P_n$ are random variables, and we prove the following

\begin{theorem}\label{theorem almost sure limit of number of points}
Let $M$ be a metric space, let $X_1,X_2, \ldots$ be i.i.d.~$M$-valued random variables, let $k\in \N_0$, and put $P_n=\{X_1,X_2, \ldots,X_n\}$.
If the distribution of the $X_i$ has support equal to a compact subset $C \subset M$, then
\[
\beta_k(C) = \lim_{n\to \infty} \beta_k(P_n) \ \ \textrm{ almost surely.}
\]
\end{theorem}
This theorem immediately results from the following lemma by continuity of the barcode map, see Proposition \ref{proposition barcodes for compact metric spaces}. This statement is a `folk theorem', and a variation of it with extra assumptions appears in the work of Cuevas and Fraiman~\cite{10.2307/2959033}. We include it here for completeness because we could not find this precise statement in the literature.
\begin{lemma}\label{lemma almost sure limit of number of points}
Let $M$ be a metric space, let $X_1,X_2, \ldots$ be i.i.d.~$M$-valued random variables, and put $P_n=\{X_1,X_2, \ldots,X_n\}$.
If the distribution of the $X_i$ has support equal to a compact subset $C \subset M$, then
\[
 \lim_{n\to \infty}d_H(C,P_n)= 0 \textrm{ almost surely.}
\]
\end{lemma}
\begin{proof}
As $\supp(X_i) = C$ we have $P_n \subset C$ with probability $1$. Thus, $$d_H(C,P_n)= \inf \left\{\veps >0 \mid C \subset B_\veps(P_n)\right\}.$$ 
 By construction, $P_n \subset P_{n+1}$ almost surely for all $n$ so that
\[
 d_H(C,P_{n+1}) \leq d_H(C,P_n) \textrm{ almost surely},
\]
and $0 \leq \lim_{n\to \infty} d_H(C,P_n)$ exists almost surely due to monotonicity. 
It thus suffices to show that $d_H(C,P_n) \to 0$ in probability. 
Here we use the property that if $Z_n \to Z$ in probability and $Z_n \to Y$ almost surely, then $Z=Y$ almost surely. 
For $\gamma > 0$ let us denote the event
\[
 A^n_\gamma = \left\{ d_H(C,P_n) > \gamma \right\}.
\]
We need to show that $\P(A^n_\gamma) \to[n\to \infty] 0$ for all $\gamma > 0$. 
Let us fix some $\gamma>0$. 
We have
\begin{equation}\label{eq agamman}
 A_\gamma^n = \left\{ C \not\subset (P_n)_\gamma \right\} = \left\{ \exists y \in C : y \not\in (P_n)_\gamma \right\}= \left\{ \exists y \in C : B_\gamma(y) \cap P_n = \emptyset \right\}.
\end{equation}
Since $C$ is compact, it is totally bounded, i.e., for each $\veps > 0$ we can find $c_1,\ldots,c_{N(\veps)} \in C$ 
such that $C \subset \bigcup^{N(\veps)}_{i=1} B_\veps(c_i)$.
For $\veps=\frac{\gamma}{2}$ it must be that
\[
 A_\gamma^n \subset \bigcup_{i=1}^{N\left(\frac{\gamma}{2}\right)} \left\{ B_{\frac{\gamma}{2}}(c_i) \cap P_n = \emptyset\right\} \textrm{ almost surely}
\]
from \eqref{eq agamman}. 
Indeed, if $\xi \in C$ is a random point satisfying $B_\gamma(\xi)\cap P_n=\emptyset$, 
then for $i \leq N\left(\frac{\gamma}{2}\right)$ such that $\xi\in B_{\frac{\gamma}{2}}(c_i)$
 we must have $B_{\frac{\gamma}{2}}(c_i)\cap P_n=\emptyset$  (otherwise we could find a point in $P_n$ at distance smaller than 
$\gamma$ from $\xi$ by the triangle inequality). 
Since the random points $X_j$ are i.i.d., we have for each $i$
\[
 \P\left( \left\{ B_{\frac{\gamma}{2}}(c_i) \cap P_n = \emptyset\right\} \right) = \prod_{j=1}^n \left( 1- \P\left(X_j\in B_{\frac{\gamma}{2}}(c_i)\right)\right) = 
\left( 1- \P(X_1\in B_{\frac{\gamma}{2}}(c_i)\right)^n.
\]
Due to subadditivity of $\P$ we conclude
\[
 \P\left(A_\gamma^n\right) \leq \P\left( \bigcup_{i=1}^{N\left(\frac{\gamma}{2}\right)} \left\{ B_{\frac{\gamma}{2}}(c_i) \cap P_n = \emptyset\right\}\right) \leq
\sum_{i=1}^{N\left(\frac{\gamma}{2}\right) } \left( 1- \P(X_1\in B_{\frac{\gamma}{2}}(c_i))\right)^n.
\]
Each term in the finite sum on the right-hand-side goes to zero as $n\to \infty$, since all the $c_i$ were chosen in the support of the distribution of $X_1$.
Since $\gamma>0$ is arbitrary, the claim follows as noted above. 
\end{proof}

It is worthwhile emphasizing that there is no condition on the distribution of the random variables such as absolute continuity, the above result is completely general
 and vaguely reminiscent of the Glivenko-Cantelli theorem.

\section{LLN and CLT for barcodes}\label{section clt}

We deduce a law of large numbers (LLN) and a central limit theorem (CLT) for $\BCbot$-valued random variables. 
This becomes meaningful via Theorem \ref{theorem embedding l1} in Section \ref{section embedding barcodes}. 
In the context of \emph{persistence landscapes}, Bubenik \cite{Bub15} observed that LLN and CLT can be deduced from general probability theory in Banach spaces. 
In this section we mirror his approach in the present (barcode representation) context.
For a general reference on probability theory in Banach spaces we refer to the monographs by Vakhania, Tarieladze, and Chobanyan \cite{VTC87} respectively by Ledoux and Talagrand \cite{LT91}.

Let us recall the definition of the {Pettis integral}. Let $(V, \norm{\cdot}_V)$ be a Banach space. Given a probability space $(\Omega,\sF,\P)$ and a random vector $\xi:\Omega\to V$, an element $v\in V$ is called the \emph{Pettis integral} of $\xi$ if for each continuous linear functional $\vphi:V \to \R$ we have
\[
\vphi(v) = \int_\Omega \vphi\left(\xi(\omega)\right) d\P(\omega).
\]
The vector $v$ is also called the \emph{expectation} of $\xi$ and is denoted by $\E[\xi]$ or $\int_\Omega \xi(\omega) d\P(\omega)$. If $\E[\norm{\xi}_V]< \infty$, then $\E[\xi]$ exists and satisfies $\norm{\E[\xi]}_V \leq \E[\norm{\xi}_V]$, see \cite[II.3.1 (c)]{VTC87}.

\begin{theorem}[LLN for barcodes]\label{theorem lln for barcodes}
Let $T: \BCbot \to V$ be a continuous map from the space of barcodes to a separable Banach space $V$. 
Let $\{X_i\}_{i\in \N}$ be an i.i.d.~sequence of $\BCbot$-valued random barcodes such that $\E[\norm{T(X_1)}]< \infty$. 
Then the sequence of random variables $(S_n)_n$ where
\begin{equation}\label{eq mean}
S_n := \frac{T(X_1) + \ldots + T(X_n)}{n}
\end{equation}
converges almost surely to $\E[T(X_1)]$.
\end{theorem}
\begin{proof}
As the $\{X_n\}_n$ are i.i.d., so are the random variables $\{T(X_n)\}_n$.~
Thus, the theorem follows from the general theory of Banach space valued probability, see \cite[Corollary 7.10]{LT91}.
\end{proof}

Let us recall the concept of \emph{type} and \emph{cotype} of a Banach space, see e.g. \cite[II.9.2]{LT91}. A \emph{Rademacher} (or \emph{Bernoulli}) \emph{sequence} is a sequence of independent random variables with values $\pm 1$ both taken with probability $1/2$. For $1 \leq p\leq 2$ a Banach space $\left(V,\norm\cdot\right)$ is said to be \emph{of type $p$} if for every Rademacher sequence $(\veps_i)_{i\in \N}$ there exists a constant $C$ such that for all finite sequences $(x_i)$ the inequality
\[
\norm{\sum_{i}\veps_i x_i}_p \leq C\cdot  \left(\sum_i \norm{x_i}^p\right)^{\frac{1}{p}}
\]
holds. Here, $\norm\cdot_p$ is defined as follows:
\[
\norm X_p =(\int_\Omega \, ||X||^p d\P )^{\frac{1}{p}},
\]
where $(\Omega, \sF, \P)$ is the underlying probability space, and the norm $\norm{\cdot}$ is the norm of the Banach space $V$.
 Similarly, $\left(V,\norm\cdot\right)$ is said to be \emph{of cotype $q$} for $1\leq q \leq \infty$ if instead there is a constant $D$ such that
\[
 \left(\sum_i \norm{x_i}^q\right)^{\frac{1}{q}}   \leq  D\cdot \norm{\sum_{i}\veps_i x_i}_q.
\]
By  \cite[Theorem 2.1]{HJP76}, being of type $p$ is equivalent the existence of a constant $C>0$ such that 
\[
\E\left[ \norm{\sum\nolimits_{j=1}^n X_j}^p\right] \leq C \cdot  \sum_{j=1}^n E\left[\norm{X_j}^p\right] 
\] 
for all independent $X_1, \ldots, X_n$ with mean $0$ and finite $p$-th moment.

Note that every Banach space is of type $1$ and that a Hilbert space is of type $2$ and cotype $2$. It can be shown that even the converse is true, i.e. a Banach space of type  $2$ and cotype $2$ is a Hilbert space, see \cite[Theorem 1.1]{Kw}.

\begin{theorem}[CLT for barcodes]\label{theorem clt for barcodes}
Let $T: \BCbot \to V$ be a continuous map from the space of barcodes to a separable Banach space $V$ of type $2$.  Let $\{X_i\}_{i\in \N}$ be an i.i.d.~sequence of $\BCbot$-valued random barcodes such that $\E[{T(X_1)}]=0$ and $\E[\norm{T(X_1)}^2]< \infty$ and let 
$S_n$ be the $V$-valued random variable from \eqref{eq mean}. 
Then $(\sqrt{n} S_n)_n$ converges weakly to a Gaussian random variable with the covariance structure of $T(X_1)$.
\end{theorem}
\begin{proof}
Separability of $V$ implies that any probability measure on $V$ is Radon. Thus, the claim follows from~\cite[Theorem 3.6]{HJP76}. 
\end{proof}

We will show next that for important classes of examples the hypotheses of Theorem~\ref{theorem lln for barcodes} and Theorem~\ref{theorem clt for barcodes} are fulfilled. 
Let $M$ be a metric space. For a finite set $P \subset M$ recall that $\beta_k(P)$ is its $k$-th barcode, see Definition \ref{definition barcode map}. For a compact set $K\subset M$, the barcode $\beta_k(K)$ is defined in Proposition \ref{proposition barcodes for compact metric spaces}. 

\begin{theorem}\label{theorem hypothesis of lln and clt for compact}
Let $M$ be a metric space and let $X$ be a random variable with values in a compact set $\sK \subset K(M)$. 
Let $T:\BCbot \to V$ be a continuous map to a separable Banach space $V$ of type $2$. 
Then $\norm{T(\beta_k(X))}$ has finite $n$-th moments for all $n\geq 0$ where $\beta_k$ denotes the $k$-th barcode.
\end{theorem}
\begin{proof}
The map $\beta_k$ is continuous with respect to the bottleneck (in the codomain) and the Hausdorff (in the domain) distances.
Thus, the image 
$$C=\left\{\norm{T(\beta_k(K))}\mid K \in \sK\right\} \subset \Rplus$$ is compact. 
Let $R:=\sup C < \infty$. 
If $\left(\Omega,\sF,\P\right)$ is the underlying  probability space on which $X$ is defined, 
then clearly $\|T(\beta_k(X))\|\leq R$ holds $\P$-almost surely, and in particular
\[
\E[\|T(\beta_k(X))\|^n] \leq R^n \int_{\Omega}d\P=R^n.
\]
\end{proof}

The following is our main application.

\begin{corollary}\label{example point cloud}
Let $M=\R^d$ and let $X_1, \ldots, X_n$ be random variables with values in a compact subset $W\subset \R^d$. Then $P_n=\{X_1, \ldots, X_n\}$ is a random variable with values in the compact subset $\sK=K(W) \subset K(\R^d)$. Thus, for every continuous map $T:\BCbot \to V$ as in Theorem \ref{theorem hypothesis of lln and clt for compact} the LLN and CLT (Theorems \ref{theorem lln for barcodes} and \ref{theorem clt for barcodes}) apply to a sequence of i.i.d. copies of $P_n$ for fixed~$n$.
\end{corollary}

\section{Sampling from the circle: expected barcode lengths}\label{section sphere}
We wish to consider the question of approximation by expectations (of transformations) of random barcodes, where the barcodes are obtained from i.i.d.~samples with a fixed (large) sample size. 

We first compute expectations in the context of i.i.d.~sampling in the simplest example at work - the circle $$\S^1=\{x=(x_1,x_2)\in \R^2 \mid x_1^2 + x_2^2 = 1\}$$ with 
 uniform samples. Recall that the uniform distribution on an $m$-dimensional manifold $M \subset \R^d$ of finite volume is defined by
\[
\P(A):=\frac{\vol(A)}{\vol(M)} \qquad \forall \ A\subset M \textrm{ measurable.}
\]
Here, $\vol$ is the $m$-dimensional volume of measurable subsets of $M$. 

In our study, we will more precisely focus on the length of the $\beta_1$-barcode for the unit circle\footnote{The $\beta_1$-barcode of $\S^1$ is shown to consist of at most one interval in Corollary \ref{corollary bn leq one}, thus we may speak of its length by which we just mean the length of that interval.}, and approach the question more generally in Section \ref{section approximation one}. In order to get these more precise results, we need to be more concrete on the distribution.

Recall that for a finite set $P\subset \S^1$ and $t \geq 0$ we denoted by $P_{t}$ the closed $t$-neighborhood of $P$. 
 Before allowing $P$ to be random, we deduce some general properties of deterministic~$P_{t}$.

\begin{lemma}\label{lemma homology sn}
If $t\in[0,1) $, the projection $\pi: P_{t} \to \S^1$, $v \mapsto \frac{v}{\norm v}$ is a homotopy equivalence onto its image $\pi(P_t) \subset \S^1$.
 If $t \geq 1$, then $P_{t}$ is star-shaped for $0 \in P_{t} \subset \R^{2}$. In particular, $P_{t}$ is contractible in that case.
\end{lemma}

\begin{example}\label{example points on circle}
Before we proceed to the proof of the lemma, let us illustrate what happens in two simple examples.

\begin{figure}[ht]
\begin{minipage}[b]{0.45\textwidth}
\centering
\includegraphics[scale=3.9]{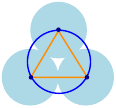}
\caption{Three points whose $t$-neighborhood has a cycle.}\label{figure three points}
\end{minipage}
\hfill
\begin{minipage}[b]{0.45\textwidth}
\centering
\includegraphics[scale=0.6]{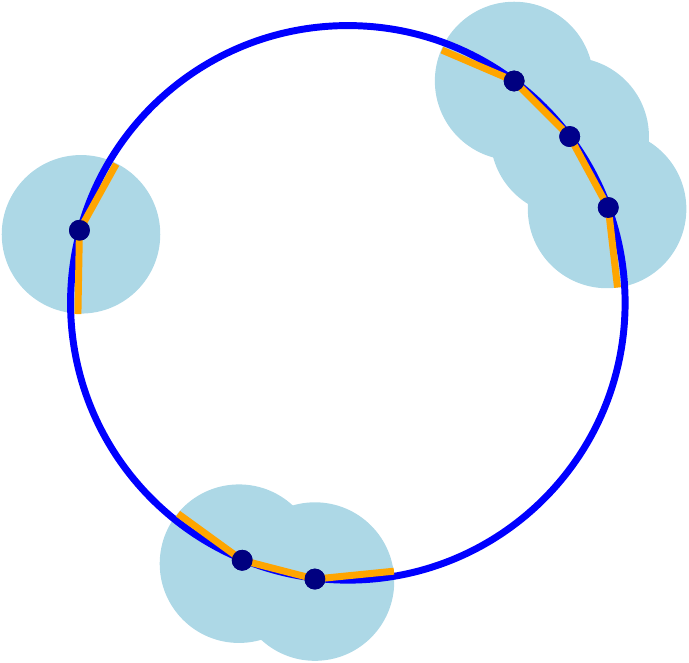}
\caption{Six points whose $t$-neighborhood has no cycle.}\label{figure six points}
\end{minipage}
\end{figure}

The first example is $P=\{x\in\C\mid x^3=1\}$ and $t=\frac{\sqrt{3}}{2}$ as depicted in Figure \ref{figure three points}. Even though $P_t$ contains a nontrivial $1$-cycle (the triangle between the three points), it does not contain $\S^1$. However, its image $\pi(P_t)$ is the full circle and indeed, $P_t$ and $\S^1$ are homotopy equivalent. The homotopy equivalence is realized by exhibiting a subspace of $P_t$ that maps homeomorphically to the sphere, namely the orange triangle. 

The second example is depicted in Figure~\ref{figure six points}. In this case both $P_t$ and $\pi(P_t)$ have three connected components and each of them is contractible. As in the previous example, the homotopy equivalence is shown by noting that the orange polygonal chain  inside $P_t$ maps homeomorphically to $\pi(P_t)$. This chain is obtained by considering each connected component of $P_t$ separately and within such a component connecting every point of $P$ through a straight line segment with its left and right neighbor (if existent) and furthermore connecting the ``leftmost'' and the ``rightmost'' point (call them $x_\ell$ and $x_r$) via a straight line segment to the unique leftmost point on the boundary of the $t$-ball around $x_\ell$ respectively to the unique rightmost point on the boundary of the $t$-ball around $x_r$. 

As explained in Example \ref{example tame} and Section \ref{section metric spaces}, the homotopy type of an ``inflated point cloud'' $P_t \subset \R^d$ can be calculated using the nerve theorem. The homology of $P_t$ is the same as the homology of the \v{C}ech complex $\check{C}_t(P)$ with parameter $t \geq 0$. The \v{C}ech filtration also gives a computational tool to get one's hand on the persistent homology of a finite point cloud. However, it turns out that there is no actual homology computation to be done in this section, because by Corollary \ref{corollary bn leq one} below the persistent homology of a finite point cloud on the circle will be rather simple.
\end{example}

\begin{proof}[Proof of Lemma \ref{lemma homology sn}]
The statement for $t\geq 1$ is clear because every point in $P_t$ is contained in a convex ball containing $0\in P_t$ with center on the circle. We will therefore assume that $t < 1$ from now on. Let us construct a homotopy inverse to $\pi$. 

As was anticipated in the examples, the homotopy equivalence will be obtained by exhibiting a subspace $G \subset P_t$ which under $\pi$ maps homeomorphically onto $\pi(P_t)$. The homotopy inverse to $\pi$ will then be $\iota:=\left(\pi\vert_G\right)^{-1}:\pi(P_t) \to G \subset P_t$ to the effect that $\pi\circ \iota = \id_{\pi(P_t)}$ and $\iota\circ \pi$ will be homotopic to the identity on $P_t$ via the homotopy $(x,t) \mapsto tx + (1-t)\iota(\pi(x))$. Observe that with a point $x$ also the line segment between $x$ and $\pi(x)$ is in $P_t$ so that the homotopy is well-defined.

First note that every connected component of $P_t$ is closed and maps onto a closed interval $I\subset \S^1$ where a closed interval on the circle is just the image of a closed interval in $\R$ under the parametrization $t\mapsto \left(\cos(t),\sin(t)\right)$. Thus, it is sufficient to treat each connected component separately. Moreover, connected components of $P_t$ are again of the form $P'_t$ for a subset $P'\subset P$ because balls $\bar B_t(x)$ are connected. In other words, we may assume that $P_t$ is connected.

If $\pi(P_t)=\S^1$, we put $n=\# P$ and let $G$ be the $n$-gon connecting the centers of the circles in circular order by line segments. This is the triangle in the first example from Example \ref{example points on circle} above.

Suppose now that $\pi(P_t) \not = \S^1$. Without loss of generality $1$ is not in the image of $\pi$. We write $P=\{p_1,\ldots,p_n\}$ such that $\arg(p_i) < \arg(p_{i+1})$ for all $i=1,\ldots,n-1$ where for all $z\neq 1$ we denote $\arg(z)\in(0,2\pi)$ the unique point such that $e^{i \arg(z)}=z$. Moreover, there are unique points $p_0 \in \bar B_t(p_1)$ and $p_{n+1}\in \bar B_t(p_n)$ such that 
\[
\arg(p_0)=\min\{\arg(z) \mid z \in \pi(P_t)\} \quad \textrm{and} \quad \arg(p_{n+1})=\max\{\arg(z) \mid z \in \pi(P_t)\}.
\]
Then, we define $G$ to be the polygonal chain which is the union of the line segments connecting $p_i$ and $p_{i+1}$ for all $i=0,\ldots,n$. We leave it to the reader to verify that $\pi\vert_G$ is a homeomorphism onto $\pi(P_t)$.
\end{proof}

\begin{corollary}\label{corollary bn leq one}
For every $t \in [0,1)$ and $P\subset \S^1$ finite we have 
\[
H_1(P_{t}, \mathbf k)=0 \textrm{ or } H_1(P_{t}, \mathbf k)= \mathbf k.
\]
\end{corollary}
\begin{proof}
By Lemma \ref{lemma homology sn} (whose notation we use) it suffices to show that $H_1(\pi(P_t),\mathbf k)=0$ or $H_1(\pi(P_t),\mathbf k)=\mathbf k.$
We have seen in the proof of the preceding lemma that the connected components of $\pi(P_t)$ are either all homeomorphic to closed intervals in $\R$ or $\pi(P_t)=\S^1$, whence the two cases.
\end{proof}

As usual, we denote by $\beta_k(P)$ the barcode obtained from the $k$-th persistent homology of a finite set $P \subset \R^d$. 
By Corollary \ref{corollary bn leq one} we know that the $\beta_1$-barcode of a point cloud $P \subset \S^1$ consists of at most one interval. 
We denote the \textbf{length} of this interval by $$\ell(\beta_1(P)) \in [0,1]$$ and also sometimes refer to it as the \textbf{length of the barcode}. Before stating the main result of this section, Theorem \ref{theorem expectation circle}, in its most general form, it might be instructive to consider the following special case.

\begin{proposition}\label{theorem expectation circle 3points}
Suppose that $P_3=\{X_1,X_2, X_3\} \subset \S^1$ is composed of three independent uniformly distributed points on the circle $\S^1$. Then
\[
\E[\ell(\beta_1(P_3))] = \frac{9(\sqrt{3} - 2)}{\pi^2} + 1/4.
\]
\end{proposition}
\begin{proof}
We parametrize the circle by the interval $I=(-\pi,\pi]$. 
Using the rotational symmetry we may assume that $X_1=\pi$ and that $X_2=\vth$, $X_3=\vphi$ where $\vth, \vphi$ are uniformly distributed random angles. It follows from Lemma \ref{lemma homology sn} that the time of death of the $\beta_1$-barcode is $t_d=1$. Its time of birth is
\begin{equation}\label{eq time of birth}
t_b=\begin{cases} 1 &\textrm{if } X_1,X_2,X_3 \textrm{ lie on a half circle}\\ \max\left( \frac{\abs{X_1-X_2}}{2},\frac{\abs{X_1-X_3}}{2},\frac{\abs{X_2-X_3}}{2}\right)\end{cases}
\end{equation}
where $\abs{\cdot}$ denotes the Euclidean norm. 
 We have
\[
\begin{aligned}
\abs{X_1-X_2}&=\sqrt{\left(1+\cos(\vth)\right)^2 + \sin(\vth)^2} = 2\cos\left(\frac{\vth}{2}\right),\\
\abs{X_1-X_3}&= 2\cos\left(\frac{\vphi}{2}\right), \\
\abs{X_2-X_3}&= 2\sin\left(\frac{\vth-\vphi}{2}\right).
\end{aligned}
\]
Now we wish to calculate
\[
\E[\ell]=\int_{I \times I} \left( t_d-t_b \right) d\P
\]
where $\P=\frac{1}{4\pi^2} \mu$ is the uniform measure on $I \times I$ and $\mu$ is the Lebesgue measure. We observe that $\ell=t_d-t_b=0$ whenever $X_1,X_2,X_3$ lie on a half circle in $\S^1$ by \eqref{eq time of birth}. 
Let $G \subset I \times I$ be the event that  $X_1,X_2,X_3$ do not lie on a half circle. We have 
\[
G=G_0 \cup \left( -G_0\right) \quad \textrm{where } G_0=\{(\vth,\vphi)\in I \times I \mid \vth \geq 0, \vth - \pi < \vphi < 0 \}.
\]
This event, as well as the function $\ell$, are invariant under $(\vth,\vphi)\mapsto (-\vth,-\vphi)$. Thus
\[
\begin{aligned}
\E[\ell]&= 2 \int_{G_0} \left(1 - t_b\right)d\P\\
&=\frac{1}{2\pi^2} \int_0^\pi \int_{\vth - \pi}^0 \left(1 - t_b(\vth,\vphi)\right)d\vphi d\vth.
\end{aligned}
\]
Next we divide $G_0=G_{12} \cup G_{13} \cup G_{23}$ into three subevents corresponding to whether $\abs{X_1-X_2}$, $\abs{X_1-X_3}$, or $\abs{X_2-X_3}$ is maximal.
 For example, $\abs{X_1-X_2}$ is maximal on $G_{12}=\{(\vth,\vphi)\mid 0< \vth < \frac{\pi}{3}, -\pi+2\vth < \vphi < -\vth\}$.
 Again by symmetry considerations, these events have the same probabilities, and the integrals (expectations) restricted to them have equal values, so that 
\[
\begin{aligned}
\E[\ell]&= \frac{3}{2\pi^2}\int_{G_{12}} \left(1 - t_b\right)d\mu\\
&=\frac{3}{2\pi^2} \int_0^{\frac{\pi}{3}} \int_{2\vth - \pi}^{-\vth} \left(1 - \cos\left( \frac{\vth}{2}\right)\right)d\vphi d\vth\\
&=\frac{9(\sqrt{3} - 2)}{\pi^2} + \frac 1 4
\end{aligned}
\]
as claimed.
\end{proof}

We note $\frac{9(\sqrt{3} - 2)}{\pi^2} + \frac 1 4 \approx 0,00565963600183$.
 The just made calculation can be generalized as follows.

\begin{theorem}\label{theorem expectation circle}
Suppose that $P_n=\{X_1,\ldots,X_n \} \subset \S^1$ is a random point set on the circle, i.e., $X_1, \ldots, X_n$ are independent, uniformly distributed $\S^1$--valued random variables. Then 
\[
\E[\ell\left(\beta_1(P_n)\right)] =1-\left(\sum_{k\geq 1}(-1)^{k-1}\binom{n}{k} \int_0^{\min\left(\frac{1}{2},\frac{1}{k}\right)} \pi \cos(\pi t) (1-kt)^{n-1} \,dt\right).
\]
\end{theorem}
\begin{proof}
This time we parametrize the circle by the interval $[0,2\pi]$, modulo $2\pi$. 
Let  $\Theta_i$ with values in $(0,1]$ be  specified through the identity $X_i=(\cos(2\pi\Theta_i),\sin(2\pi\Theta_i))=\exp\{2\pi i\, \Theta_i\}$.
It is again natural to identify one  of the points  (for example the last one) with the angle $0=2\pi$.
Let $\Theta_{(i)}$ be the $i$-th order statistic of $(\Theta_1,\ldots,\Theta_{n-1})$, i.e. the $i$-th smallest value among $(\Theta_1,\ldots,\Theta_{n-1})$, and let us set in addition $\Theta_{(0)}:=0$ and $\Theta_{(n)}: =1$. 
The  \emph{normalized (angular) spacings} between the points are defined as follows: $S_i:=\Theta_{(i)}-\Theta_{(i-1)}$ for $i=1,\ldots,n$. 
We also define 
$$
X_{(i)}:=\exp\{2\pi i \,\Theta_{(i)}\}, \ \ i=0,1,\ldots,n,
$$ 
so that the 2-dimensional random points are ordered via their respective angles (similarly to the proof of Lemma \ref{lemma homology sn}).

It is easy to check by induction (or alternatively look in \cite{BK07} or \cite{Dev81})
that the joint distribution of the spacings vector $(S_1,\ldots,S_n)$ is uniform on the unit $n-1$-simplex, as given by,
$$
\P(S_1>a_1,S_2>a_2,\ldots, S_n>a_n)=
\begin{cases}
(1-\sum_j a_j)^{n-1}, & \sum_j a_j<1\\
0, & \sum_j a_j\geq 1
\end{cases}
$$
As in the case of three random points above,  from Lemma \ref{lemma homology sn} we know that the $\beta_1$-barcode dies at time $t_d=1$ and is born at time
\begin{equation}\label{eq birth n points}
t_b=\begin{cases} 1, &\textrm{if } X_1,X_2,\ldots,X_n \textrm{ lie on a half circle}\\ \max_{i=1}^n |X_{(i)}-X_{(i-1)}|/{2} & \textrm{otherwise}\end{cases}.
\end{equation}
The first condition in \eqref{eq birth n points} is equivalent to the \emph{maximal spacing} $M_n:=\max_{i=1}^n S_i$ being $\geq 1/2$. 
However on $\{M_n<1/2\}$ we have
$$
\max_{i=1}^n \frac{|X_{(i)}-X_{(i-1)}|}{2} = \sin(\pi M_n).
$$
For the remainder of the calculation let us abbreviate $\ell\left(\beta_1(P_n)\right)$  by $\ell$.
Due to the just made observations we conclude that $\E[\ell] =\E\left[(1-\sin(\pi M_n))\one{\{M_n<1/2\}}\right]$. 
From the above given expression for the joint residual distribution of spacings and the inclusion-exclusion formula, one deduces the following expression for the residual distribution of $M_n$:
$$
\P(M_n>x)=\P(M_n\geq x)=\sum_{k\geq 1: \, kx <1 } (-1)^{k-1} \binom{n}{k} (1-kx)^{n-1}.
$$
This formula is attributed to Whitworth \cite{Whi97}.
Let us define $g: [0,1] \mapsto [0,1]$ as
$$
g(t):=\begin{cases} \sin{(\pi x)}, & x<1/2\\
1, & x\geq 1/2.
\end{cases}
$$
Now $\E[\ell] = 1- \E\left[g(\pi M_n)\right]$. 
Since $g$ is non-negative and differentiable (of class $C^1$), we can apply a well-known change of (order of) integration formula
$$
\E\left[g(\pi M_n)\right] = \int_{t\geq 0} g'(t) \P(M_n\geq t) \,dt =\int_0^{\frac{1}{2}} \pi \cos(\pi t) \P(M_n\geq t)\, dt ,
$$
which equals 
$$
\sum_{k\geq 1}(-1)^{k-1}\binom{n}{k} \int_0^{\min\left(\frac{1}{2},\frac{1}{k}\right)} \pi \cos(\pi t) (1-kt)^{n-1} \,dt.
$$
\end{proof}

\begin{remark}[Related work]
Similar computations to ours were made in Bubenik and Kim \cite{BK07}
in the setting of Vietoris-Rips filtration (as opposed to \v{C}ech filtration), 
and with respect to the angular (unlike Euclidean taken here) metric on points. 
\end{remark}
\section{Approximation by expected transformations of random barcodes}\label{section approximation one}
The calculations made in the previous section demonstrate that expected functionals of barcodes can be quite difficult (and, for more complicated examples, impossible) to obtain explicitly.
Theorem \ref{theorem almost sure limit of number of points} applied to $\S^1$ on the other hand tells us that as $n$ gets large, in the notation of the previous section, the length $\ell(\beta_1(P_n))$ of the single bar 
comprising $\beta_1(\S^1)$ must converge to $1$.
If interested in the asymptotics of $\ell(\beta_1(P_n))$ and $\E\left[\ell(\beta_1(P_n))\right]$, we refer the reader to Devroye~\cite{Dev81}. 
In particular, since $\sin(x)\sim x$ for small $x$, one can apply \cite{Dev81}, Lemma 2.5 saying $\frac{n M_n}{\log{n}} \to 1$ in probability, whereas in the last section $M_n$ denotes the maximal spacing. Therefore, $M_n\to 0$ 
almost surely, and $1-\ell(\beta_1(P_n))$ is of order $\frac{\log{n}}{n}$ with an overwhelming probability as $n\to \infty$.
Similar considerations based on \cite{Dev81}, Lemma 2.6 lead to $\E\left[\ell(\beta_1(P_n))\right]=1-\Theta(\frac{\log{n}}{n})$ as $n\to \infty$.

This is an interesting example that motivates the study of the quality of such an approximation in general.

Similarly to Section \ref{section approximation two}, one could consider, for a fixed (and relatively large) $n\in \N$, i.i.d.~$\R^d$-valued random variables $X_1,\ldots, X_n$, where the joint distribution has support on some compact subset $M \subset \R^d$.
The $k$-th barcode of the resulting random finite set $P_n=\{X_1,\ldots,X_n\}$ yields a random barcode $\beta_k(P_n)$ for each $k$. 
Suppose that $T:\BC \to V$ is a continuous function from the barcode space to some Banach space.
By Theorem \ref{theorem lln for barcodes} and Theorem \ref{theorem hypothesis of lln and clt for compact}, the expected value $\E[T(\beta_k(P_n))]$, can be well approximated by the
 empirical means (taken over many i.i.d.~samples of point clouds of size $n$).

We restrict our hypotheses somewhat with respect to those of Section \ref{section approximation two}, in assuming in addition that $M$ is a compact $m$-dimensional 
manifold in $\R^d$, 
and the distribution of $X_1$ above is uniform on $M$. 
We are working on relaxing these hypotheses in a forthcoming project. 
 Let us first introduce some notation. 
Recall that the \emph{medial axis} of $M$ is defined as the closure of the set of points in $\R^d$ that do not have 
a unique nearest point on $M$. 
We denote by $\tau=\inf_{p\in M} \sigma(p) $ the infimum of the distances $\sigma(p)$ of $p\in M$ from the medial axis of $M$, i.e., every point in the open $\tau$-neighborhood has a unique nearest point on $M$. 
It follows from compactness that $\tau$ is positive. The quantity $\tau$ is referred to as 
the \emph{reach} of $M$. 
 
Under the above assumptions,  we can rely on the work by Niyogi et al.~\cite{NSW08}.
The result \cite{NSW08}, Theorem 3.1 is not sufficient for our purposes, therefore we prove a stronger statement
in Theorem \ref{theorem NSW reinforced} and explain how this also follows from the analysis in \cite{NSW08}, see also Remark \ref{remark nsw reinforced}.
Let
\begin{equation}\label{eq nsw constants}
\begin{aligned}
c_1 (\eps)&:=\frac{\vol(M) } {\cos\left({\arcsin{\left(\frac{\eps}{8\tau}\right)} }\right)^m \vol\left(B^m_{\eps/4}(0)\right)} ,  \\ \\
c_2 (\eps)&:=\frac{\vol(M) } {\cos\left({\arcsin{\left(\frac{\eps}{16\tau}\right)}}\right)^m \vol\left(B^m_{\eps/8}(0)\right)},
\end{aligned}
\end{equation} 
where the superscript $m$ indicates that the balls of radii $\eps/4$ and $\eps/8$, respectively, are taken in $\R^m$ (and not necessarily in the ambient space $\R^d$). 
In particular, the smaller the $\eps$, the larger are $c_{1, 2}$, and they are of order $1/\eps^m$. We will use these constants throughout this section.

Let $A\subset \R^d$ be a set and $t\geq 0$. As in Section \ref{section barcodes} we denote by $A_{t}$ the closed $t$-neighborhood of $A$. For every manifold $M$ with reach $\tau$ as above and for every $0\leq t < \tau$ the inclusion $\iota_t: M\into M_{t}$ is a homotopy equivalence. This is almost by definition of the reach: a homotopy inverse is given by the projection $\pi: M_{t}\to M$ to the nearest point on $M$. Note that for any $p\in M_{t}$ the line segment connecting $p$ to $\pi(p)$ is entirely contained in $M_{t}$ (even in the fiber of $\pi$ over $\pi(p)$) so that a simple convex combination between $\iota \circ \pi$ and the identity gives a homotopy equivalence. For $A\subset M$ and $t \in [0,\tau)$ we denote
\begin{equation}\label{eq homotopy equivalence}
\chi_{A,t} : A_{t} \into M_{t} \to[\pi] M
\end{equation}
the composition of the inclusion with the projection.

\begin{theorem}\label{theorem NSW reinforced}
Let $M \subset \R^d$ be a smooth compact submanifold of dimension $m$ and let $X_1,X_2,\ldots,X_n$ be an i.i.d. 
random sample from $M$ for the uniform distribution. Denoting $P_n:=\{X_1,\ldots,X_n\}$ we have that
if $\veps \in (0,\sqrt{\frac{3}{5}}\tau)$, then for each $\delta>0$ and each
\begin{equation}\label{eq nsw n}
n>c_1(\veps)\left(\log(c_2 (\veps))+\log{\dfrac{1}{\delta}}\right),
\end{equation}
 the map $\chi_{P_n,t} : (P_n)_{t}\to M$ from \eqref{eq homotopy equivalence} is a homotopy equivalence for all $t \in \left[\veps,\sqrt{\frac{3}{5}}\tau\right)$ with probability at least $1-\delta$. 
\end{theorem}
\begin{remark}\label{remark nsw reinforced}
We could have restricted $\delta$ to $(0,1]$, but prefer this statement (trivially true if $\delta>1$ since any probability is 
non-negative) in view of applications below.
A careful comparison with \cite{NSW08}, Theorem 3.1, reveals several differences, but only one is responsible for the fact that the just stated result is non-trivially stronger in the stochastic sense. 
The claim in Theorem \ref{theorem NSW reinforced} is that for any $0 \leq \veps < \sqrt{\frac{3}{5}}\tau$ the map $\chi_{P_n,t} : {P_n}_{t}\to M$ from \eqref{eq homotopy equivalence} is a homotopy equivalence \emph{on the whole interval} of parameters $t \in [\veps,\sqrt{\frac{3}{5}}\tau)$   on one and the same event of a sufficiently large probability. 
The claim in \cite{NSW08} is only that $\chi_{P_n,\veps} $ is a homotopy equivalence at the given parameter $\veps$  on an event of a sufficiently large probability.
However, an intersection of many (let alone, infinitely many) highly probable events may have a drastically smaller probability. 
This does however not happen here, for the reasons we give next. We do not contribute any new argument for this, the stronger formulation stated in Theorem \ref{theorem NSW reinforced} is merely a consequence of ordering the arguments of \cite{NSW08} accordingly.
\end{remark}
\begin{proof}[Proof of Theorem \ref{theorem NSW reinforced}]
Recall that $P_n$ is called \emph{$\veps$-dense} if the open $\veps$-neighborhood of $P_n$ covers $M$. For a given $\veps \in (0,\sqrt{\frac{3}{5}}\tau)$, $\delta >0$,  and $n$ satisfying \eqref{eq nsw n}, the event $A_\veps$ defined by the random point cloud $P_n \subset M$ being \emph{$\frac{\veps}{2}$-dense} in $M$, has probability at least $1-\delta$ by Lemma 5.1 in \cite{NSW08}. Therefore, on the same event $A_\veps$ the same point cloud is $\frac{t}{2}$-dense for every $t\in [\veps,\sqrt{\frac{3}{5}}\tau)$. 

Now we infer from Proposition 3.1 in \cite{NSW08} the deterministic statement that whenever a subset $P\subset M$ is $\frac{t}{2}$-dense, the map $\chi_{P,t}:P_{t} \to M$ is a homotopy equivalence. Let $(\Omega,\sF,\P)$ denote the corresponding probability space. Then we apply the above reasoning and the just mentioned proposition to obtain
$$
\begin{aligned}
A_\veps &=\left\{ \omega \in \Omega\middle| P_n(\omega) \textrm{ is } \frac{\veps}{2}\textrm{-dense} \right\} \\
&=\left\{ \omega \in \Omega\middle| P_n(\omega) \textrm{ is } \frac{t}{2}\textrm{-dense for all } t\in \left[\veps, \sqrt{\frac{3}{5}}\tau\right) \right\} \\
&=\left\{ \omega \in \Omega\middle| \chi_{P_n(\omega),t} \textrm{ is a homotopy equivalence for all } t\in \left[\veps, \sqrt{\frac{3}{5}}\tau\right) \right\},\\
\end{aligned}
$$
Together with Lemma 5.1 from \cite{NSW08} for these sets $A_\veps$, the claim follows. Note that the quantities from that Lemma 5.1 are bounded  
according to the analysis in section 5 of \cite{NSW08} in such a way that \eqref{eq nsw n} holds.
\end{proof}

We will make essential use of the following easy but important observation.

\begin{lemma}\label{lemma barcode approximation}
Let $M \subset \R^d$ be a smooth compact submanifold of dimension $m$ and reach $\tau$, let $X_1,X_2,\ldots,X_n$ be an 
 i.i.d. random sample from $M$ for the uniform distribution, and put $P_n:=\{X_1,\ldots,X_n\}$.
 Then for each $\veps \in (0,\sqrt{\frac{3}{5}}\tau)$, each $\delta>0$, and each
\begin{equation}\label{eq nsw n2}
n>c_1(\veps) \left(\log(c_2 (\veps))+\log{\dfrac{1}{\delta}}\right),
\end{equation}
we have that
\[
d_\infty\left(\beta_k(M),\beta_k(P_n)\right) \leq \frac{\veps}{2}
\]
with probability at least $1-\delta$.
\end{lemma}
\begin{proof}
By \cite[Proposition 3.2]{NSW08} for every $\veps \in (0,\sqrt{\frac{3}{5}}\tau)$ the sample $P_n$ is 
$\frac{\veps}{2}$-dense in $M$. Because of $P_n\subset M$ this just means that $d_H(P_n,M)\leq \frac{\veps}{2}$ for the Hausdorff distance $d_H$. Therefore, the claim follows by \eqref{proposition barcodes for compact metric spaces item one} of Proposition \ref{proposition barcodes for compact metric spaces}.
\end{proof}

To formulate our next result, we introduce an operator on barcodes. 
 For any two $a,b$ such that $0<a\leq b<\infty$, let $R_{[a,b]}$ denote the restriction map 
$R_{[a,b]} :\BCbot \to \BCbot$ defined as follows:
for each finite barcode representation $b=\{I_i\}_{i=1}^n \in \BC$ with $I_i =(x_i,d_i)\in \R_{\geq 0}^2$ we first define
$$I_i^{|(a,b)}:=\left(\max(x_i,a), \min(x_i+d_i,b)-\max(x_i,a)\right)$$ if $\min(x_i+d_i,b)\geq \max(x_i,a)$ and $I_i^{|(a,b)}:=(x_i,0)$ otherwise. Finally, we put $R_{[a,b]}(b)= \{I_i^{|(a,b)}\}_{i=1}^n$. 
Since the thus defined $R_{[a,b]} :\BC \to \BC$ is clearly a 1-Lipshitz map, 
we can extend it as usual to $R_{[a,b]} :\BCbot \to \BCbot$.
Note that the coordinates of $I_i^{|(a,b)}$ are just the starting point and the length of the  interval 
$[x_i,x_i+d_i]\cap [a,b]$ if nonempty. 

For further use we also record that for a given barcode $\beta \in \BCbot$ the barcode $R_{[a,b]}(\beta)$ depends continuously on $a$ and $b$.

\begin{setup}\label{setup banach space} We fix a Lipschitz continuous map $T:\BCbot \to V$ to some Banach space $\left(V,\norm\cdot\right)$ with Lipschitz constant $L(T)>0$. 
Due to compactness and continuity, the transformed barcodes $T(\beta_k(M))$ and $T(\beta_k(P_n))$ are uniformly bounded over $n$ by some finite number, which we denote by $C(M;T)$. 
 We also know that, for large $n$, both $T(\beta_k(P_n))$ and $\E[T(\beta_k(P_n))]$ (due to the dominated convergence theorem) approximate $T(\beta_k(M))$. 
 The question is how large can the difference of $T(\beta_k(M))$  and $\E[T(\beta_k(P_n))]$ be? 
By interpreting Theorem \ref{theorem NSW reinforced} we arrive to the following conclusion. 
\end{setup}

\begin{theorem}\label{theorem consequence of NSW}
Let $M \subset \R^d$ be a smooth compact submanifold of dimension $m$ and reach~$\tau$, 
 let $X_1,X_2,\ldots,X_n$ be an i.i.d. 
random sample from $M$ for the uniform distribution, and denote $P_n:=\{X_1,\ldots,X_n\}$. Let $\veps\in \left[0,\sqrt{\frac{3}{5}}\tau\right)$, and put $I_\veps:= \left[\veps,\sqrt{\frac{3}{5}}\tau\right)$. Then for all $k \in\N_0$ the following hold:
\begin{enumerate}
	\item\label{theorem consequence of NSW item one} Let $\bar I_\veps=\left[\veps,\sqrt{\frac{3}{5}}\tau\right]$ denote the closure of the 
interval $I_\veps$ and $c_1(\veps)>0$ and $c_2(\veps)>0$ be as in \eqref{eq nsw constants}. Then:
\[
{\E\left[\norm{T\circ R_{\bar I_\veps}(\beta_k(P_n)) - T\circ R_{\bar I_\veps}(\beta_k(M))}_V\right] } \ \leq \ 3c_2(\veps) \exp\left(\frac{-n}{c_1(\veps)}\right) C(M;T).
\] 
\item\label{theorem consequence of NSW item two} For the unrestricted barcodes we have:
\[
{\E\left[\norm{T\left(\beta_k(P_n)\right)- T\left(\beta_k(M)\right)}_V\right] } \ \leq \ 3c_2(\veps) \exp\left(\frac{-n}{c_1(\veps)}\right) C(M;T) \ + \ \frac{L(T)\cdot \veps}{2}.
\] 
\end{enumerate}
Here, $T$, $C(M;T)$ and $L(T)$ are as in Setup \ref{setup banach space}.
\end{theorem}                                     
\begin{proof}
Let us prove \eqref{theorem consequence of NSW item one}. By continuity of the projection as a function of the (endpoints of) the interval, it suffices to prove the inequality for every closed interval contained in $I_\veps$. Let $I \subset I_\veps$ be such an interval.

Due to Theorem \ref{theorem NSW reinforced}, with our choice of $n$ we have that for all $s\in I_\veps$ the homology of 
$M$ equals that of the point cloud thickened by $s$, except on an event $E_\veps$ of probability at most $\delta$. 
Condition \eqref{eq nsw n} is equivalent to 
\[
\delta > c_2(\veps)  \exp{\left(-\frac{n}{c_1(\veps)}\right)}. 
\]
In particular, we could take $\delta(n)=3c_2(\veps)  \exp{\left(-\frac{n}{c_1(\veps)}\right)}/2$. 
Therefore, we find $\P(E_\veps) \leq \frac{3}{2} c_2(\veps)  \exp{\left(\frac{-n}{c_1(\veps)}\right)}$, 
and on the complement of $E_\veps$ we know that the  homology of the inflated point cloud $P_s$ does 
not change when $s\in I$ varies, and is equal to that of $M$ and hence to that of $M_s$. 

In particular, $T \circ R_{I}(\beta_k(P_n)) =T\circ R_{I}(\beta_k(M))$ for all 
$k\in \N_0$ on the complement $E_\veps^c$. 
To arrive at the above stated bound, for each given $k$, we apply the trivial upper bound 
$\norm{T \circ R_{I}(\beta_k(P_n)) -T\circ R_{I}(\beta_k(M))}_V \leq 2 C(M;T)$ on $E_\veps$, 
and take expectation.

For the proof of \eqref{theorem consequence of NSW item two}, we just have to note that on $E_\veps^c$ we have $d_\infty(\beta_k(P_n),\beta_k(M)) \leq \frac{\veps}{2}$ by Lemma \ref{lemma barcode approximation}. The claim follows as $T$ is $L(T)$-Lipschitz and $\P(E_\veps^c)\leq 1$.
\end{proof} 

In particular, the theorem implies that the quantity $\norm{\E\left[T(\beta_k(P_n))\right]-T(\beta_k(M))}_V$ satisfies the same inequalities as in Theorem \ref{theorem consequence of NSW} thanks to the basic properties of the Pettis integral, see Section~\ref{section clt}.


\section{Embedding the space of barcodes}\label{section embedding barcodes}

In Section \ref{section clt} we have deduced LLN and CLT for random variables induced from random barcodes. We have been working with Lipschitz continuous maps from $\BCbot$ to some Banach space. In this section we will take a look at one such example by building on work of the first named author \cite{Kal18}. Let $\BC$ denote the space of barcode representations.  Our goal is to describe a Lipschitz continuous embedding $\BCbot \into \ell_1$.

Let us consider the operations $\boxplus,\oplus,\odot$ on $\R$ defined as 
\[
a\oplus b := \min{(a, b)}, \quad  a\boxplus b := \max{(a, b)}, \quad  a\odot b := a+b.
\]
We call $(\R, \boxplus, \odot)$ the max-plus semiring and $(\R, \oplus, \odot)$ the tropical semiring.

Just as ordinary polynomials are formed by multiplying and adding real variables, max-plus polynomials can be formed by multiplying and adding variables in the max-plus semiring. Let $x_1, x_2, \ldots, x_N$ be variables that represent elements in the max-plus semiring.  A \emph{max-plus monomial expression} is any product of these variables, where repetition is permitted. By commutativity, we can sort the product and write monomial expressions with the variables raised to exponents:
\[
p(x_1, x_2, \ldots, x_N) = a_1\odot x_1^{a_1^1} x_2^{a_2^1} \ldots x_N^{a_N^1} \boxplus a_2\odot x_1^{a_1^2} x_2^{a_2^2} \ldots x_N^{a_N^2}\boxplus \ldots \boxplus a_m\odot x_1^{a_1^m} x_2^{a_2^m} \ldots x_N^{a_N^m}.
\]
Here the coefficients $a_1, a_2, \ldots a_m$ are in $\R$, and the exponents $a_j^i$ for ${1\leq j \leq N}$ and ${1\leq i \leq m}$ are in $\N_0$. 

Different max-plus polynomial expressions may happen to define the same functions. Thus, if $p$ and $q$ are max-plus polynomial expressions and
\[
p(x_1, x_2, \ldots, x_N) = q(x_1, x_2, \ldots, x_N)
\]
for all $(x_1, x_2, \ldots, x_N)\in \R^N$, then $p$ and $q$ are said to be \emph{functionally equivalent}, and we write $p \sim q$.  Max-plus polynomials are the semiring of equivalence classes of max-plus polynomial expressions with respect to $\sim$.  

The goal of \cite{Kal18} was to identify sufficiently many max-plus polynomials on $\BC$ to separate points. This involves finding functions invariant under the action of the symmetric group. To be able to list these functions, consider the set $\scrE_N$ of $(N\times 2)$-matrices with entries in $\{0,1\}$. The symmetric group $S_N$ acts on $\scrE_N$ by permuting the rows. To a matrix $E=(e_{i,j})_{i,j} \in \scrE_N$ we associate the max-plus monomial $P(E) = x_{1, 1}^{e_{1,1}} x_{1, 2}^{e_{1,2}} \ldots x_{N, 1}^{e_{N,1}} x_{N, 2}^{e_{N, 2}}$. Suppose that the $S_N$-orbit of $E$ is  $[E]= \{ E_1, E_2,\ldots, E_m\}$. Then $P_E= P(E_1)\boxplus P(E_2)\boxplus \ldots \boxplus P(E_m)$ is a 2-symmetric max-plus polynomial and a we can define a function $P_{k, E}$ on $\BC_n$ as 
\begin{equation}\label{poly}
P_{k,E}(x_1,d_1,\ldots,x_N,d_N):=P_{E}(x_1\oplus d_1^k,d_1,\ldots,x_N\oplus d_N^k,d_N).
\end{equation}

For $m,n \in \N_0$ with $m+n \geq 1$ we denote by $E_{m, n}$ the matrix 
\[
\begin{array}{c@{\!\!\!}l}
  \left( \begin{array}[c]{cc}
	1 & 1\\
	\vdots & \vdots\\
	1 & 1\\
	0 & 1\\
	\vdots&\vdots\\
	0 & 1\\
  \end{array}  \right)
&
 \begin{array}[c]{@{}l@{\,}l}
\left. \begin{array}{c} \vphantom{0} \\ \vphantom{\vdots}
   \\ \vphantom{0}  \end{array} \right\} & \text{$m$ times} \\
\left. \begin{array}{c} \vphantom{0}
  \\ \vphantom{\vdots} \\ \vphantom{0} \end{array} \right\} & \text{$n$ times} \\
\end{array}
\end{array}
\]
and write $P_\klmn$ for the polynomial $P_{k, E_{m,n}}$. This is a function on $\BC$; if $b$ is a barcode with $N$ bars, then 
\begin{itemize}
\item
if $m+n=N$, we use Equation~(\ref{poly});
\item
if $m+n>N$, then we add $N-m-n$ 0 length bars to $b$ and then use Equation~(\ref{poly});
\item
if $m+n<N$, then we add $N-m-n$ 0 length rows to the $E_{m, n}$ matrix and then use Equation~(\ref{poly}) for this matrix.
\end{itemize}
It was shown in \cite[Theorem 6.7]{MKV2016} that the set of functions $\{P_{\klmn}\}_{\klmn \in \N_0^3}$ separates points from $\BC$.  Furthermore, all of these functions are Lipschitz~\cite{Kal18}, i.e.\ for $C(k, m, n)=2(2m\max(k, 1) + 2m + 2n)$, the estimate
\begin{equation}\label{lipschitz}
|P_{\klmn}(b)-P_{\klmn}(b')|\ \leq \ C(k,m,n)\ d_B(b, b')
\end{equation}
holds for $b, b'\in \BC$. 

We fix once and for all an enumeration $(k_1,m_1,n_1),(k_2,m_2,n_2),\ldots$ and consider the corresponding coordinates on the barcode space. We obtain:

\begin{theorem}\label{theorem embedding l1}
The sequence $(\frac{1}{C(k_t, m_t, n_t)t^2}P_{k_t,m_t,n_t})_{t\in \N}$ of functions $\BC \to \R$ defines an injective map $\iota:\BC \into \lone$. This map is Lipschitz continuous.
\end{theorem}
\begin{proof}
Let $b\in \BC$ be a barcode. We will first prove that $\iota(b)$ is well-defined, i.e., lies in $\lone$. Let us write $b=(x_1, d_1, \ldots, x_N, d_N)$ for $x_i, d_i \in \R^2_{\geq 0}$, and let $M:=\max_{i=1}^N \max(x_i,d_i)$. For any $k,m,n \in \N_0$ we claim that $\frac{1}{C(k_t, m_t, n_t)}P_\klmn(b) \leq 2MN$. Since $P_\klmn(b)$ is the maximum of $P(E)(b)$ where $E$ runs through the orbit of $E_{m,n}$ and the monomials $P(E)$ have degree $2N$, 
\[
P(E)(x_1\oplus d_1^k,d_1, \ldots, x_N\oplus d_N^k,d_N) \leq P(E)(x_1,d_1, \ldots, x_N,d_N) \leq P(E)(M,\ldots,M) \leq 2NM.
\] 
Since $C(k,m,n)\geq 1$, $\frac{1}{C(k_t, m_t, n_t)}P_\klmn(b) \leq 2MN$. 
Consequently,
\[
\sum_{t\in \N} \frac{1}{C(k_t, m_t, n_t)t^2}\abs{P_{k_t,m_t,n_t}(b)} \leq \sum_{t\in \N} \frac{2MN}{t^2} < \infty.
\]
As mentioned above the functions $P_\klmn$ separate points on $\BC$ so that $\iota$ is indeed injective. 

This embedding is Lipschitz since it follows from Equation~\eqref{lipschitz} that
\[
\begin{aligned}
\sum_{t=1}^\infty \left|(\frac{1}{C(k_t, m_t, n_t)t^2}P_{k_t,m_t,n_t})(b)-(\frac{1}{C(k_t, m_t, n_t)t^2}P_{k_t,m_t,n_t})(b')\right|&\leq \sum_{t=1}^\infty \frac{1}{t^2} d_B(b, b')\\ &=\frac{\pi^2}{6} d_B(b, b').\\
\end{aligned}
\]
\end{proof}

\begin{example}
It is easy to see that the scaling by $\frac{1}{t^2}$ in the definition of $\iota$ is necessary. Consider e.g. $b=(1,1,\ldots,1,1)\in\BC_N \subset \BC$. Then the sequence $a_\ell=P_{0,\ell,0}(b) =2\ell$ if $\ell\leq N$ and $a_\ell=P_{0,\ell,0}(b) =2N$ otherwise. In particular, $\sum_\ell a_\ell$ diverges.
\end{example}

\begin{remark}
Note that for $1\leq p\leq q \leq \infty$ we have $\ell_p \subset \ell_q$ and the inclusion is Lipschitz continuous. In particular, we have a Lipschitz continuous embedding $\BCbot$ into the separable Hilbert space $\ell_2$, thus into a separable Banach space of type $2$ as in the assumptions of several results in Section \ref{section clt}.
\end{remark}

\section{Discussion}\label{section discussion}

The focus in the present paper is on perfect data, sampled without noise. 
It seems important to allow for noise, and therefore for data issued from distributions with unbounded support. 
Once we allow for noise (potentially with unbounded support), with the number of points $n$ being large, the maximal error will typically also be large with high probability.
To overcome this problem, it seems reasonable to assume that, for each $n$, the random points are sampled independently from a distribution indexed by $n$,
 in such a way that the maximal error stays bounded in $n$ with high probability. We postpone this study to a future work.

\bibliography{CLT}
\bibliographystyle{alpha}

\end{document}